\documentclass[11pt]{article}

\usepackage[a4paper,margin=1in]{geometry}
\usepackage{amsmath,amssymb,amsthm,mathtools,bm}
\usepackage[shortlabels]{enumitem}
\usepackage{microtype}
\usepackage{xcolor}
\usepackage{mathrsfs} 
\numberwithin{equation}{section}

\usepackage{authblk}

\usepackage[numbers,sort&compress]{natbib}
\theoremstyle{plain}
\newtheorem{theorem}{Theorem}[section]
\newtheorem{proposition}{Proposition}[section]
\newtheorem{corollary}{Corollary}[section]

\theoremstyle{definition}
\newtheorem{definition}{Definition}[section]
\newtheorem{assumption}{Assumption}[section]

\theoremstyle{remark}
\newtheorem{remark}{Remark}[section]

\newcommand{\R}{\mathbb{R}}
\newcommand{\N}{\mathbb{N}}

\newcommand{\ip}[2]{\left\langle #1, #2\right\rangle}
\newcommand{\ceq}{:=}
\newcommand{\matrixnorm}[1]{\left| #1 \right|}
\DeclareMathOperator{\linspan}{span}
\DeclareMathOperator{\dom}{dom}
\DeclareMathOperator{\supp}{supp}
\newcommand{\dualu}{u}
\newcommand{\dualv}{v}
\newcommand{\momsupport}{\dom f}
\DeclareMathOperator{\Id}{Id}

\usepackage[colorlinks]{hyperref}
\hypersetup{linkcolor=blue, citecolor=blue, urlcolor=blue}
\usepackage[nameinlink,capitalize]{cleveref}

\crefformat{equation}{(#2#1#3)}
\crefrangeformat{equation}{(#3#1#4) to~(#5#2#6)}
\crefmultiformat{equation}{(#2#1#3)}{ and~(#2#1#3)}{, (#2#1#3)}{ and~(#2#1#3)}

\newcommand{\ba}{a}
\newcommand{\bb}{b}

\newcommand{\bp}{p}

\newcommand{\br}{r}

\newcommand{\bx}{\mathbf{x}}
\newcommand{\by}{y}
\newcommand{\bz}{z}
\newcommand{\np}{N}
\newcommand{\fp}{f_{\rm mom}}
\newcommand{\fd}{f_{\rm sos}}

\newcommand{\fpc}{f_{pc}}

\newcommand{\id}{I} 
\DeclareMathOperator{\tr}{tr}

\DeclareMathOperator{\rank}{rank}

\begin{document}

	\title{Polyconvexity with Moments and Sums of Squares}
	\author[1]{Giovanni Fantuzzi}
    \author[2,3]{Didier Henrion}
    \author[6,7]{Martin Kru\v z\'{\i}k}
    \author[1]{Ajay Murali}
    \author[3]{Stephan Weis}
    
    \affil[1]{Dept. Mathematics, Friedrich-Alexander-Universit\"at Erlangen-N\"urnberg, Germany}
    \affil[2]{LAAS-CNRS, University of Toulouse, France}
    \affil[3]{Faculty of Electrical Engineering, Czech Technical University in Prague, Czechia}
    \affil[6]{Institute of Information Theory and Automation, Czech Academy of Sciences, Czechia}
    \affil[7]{Faculty of Civil Engineering, Czech Technical University in Prague, Czechia}

	\date{\today}
	\maketitle

	\begin{abstract}
    A function of a matrix is polyconvex when it can be expressed as a convex function of the matrix minors. Polyconvexity is a regularity condition ensuring existence of minimizers in nonlinear elasticity and, more broadly, in vectorial problems of the calculus of variations, when minimizing integral gradient functionals. The polyconvex envelope of a function is the largest polyconvex lower bound. Yet deciding whether a given energy is polyconvex, or computing the polyconvex envelope, are generally difficult problems. This paper focuses on polynomial matrix functions. We propose (i) tractable convex-optimization based sufficient conditions to certify polyconvexity via sum-of-squares (SOS) technology, and (ii) a principled numerical method to compute the polyconvex envelope pointwise, based on the moment-SOS hierarchy from polynomial optimization.
	\end{abstract}

\section{Introduction}

Many problems in continuum mechanics, especially nonlinear elasticity, can be modelled through the minimization of integral functionals such as
\begin{equation*}
	I(\by)=\int_\Omega f(\nabla \by(\bx))\, d\bx
\end{equation*}
over functions $\by$ in a Sobolev space $W^{1,p}(\Omega;\R^m)$, $1<p<+\infty$, subject to  prescribed boundary conditions \cite{Pedregal2000,Rindler2018,dacorogna08,ciarlet88,kruzik-roubicek19}. Here, $\Omega\subset\R^n$ is a bounded Lipschitz domain and $f:\R^{m\times n}\to \R \cup \{+\infty\}$ is a given function.

The existence of optimizers relies on the properties of $f$. In particular, it is well-known \cite{dacorogna08,Pedregal2000,Rindler2018} that minimizers exist in $W^{1,p}(\Omega;\R^m)$ if $f$ satisfies suitable growth and coercivity conditions and, crucially, is \emph{quasiconvex} in the sense of Morrey \cite{Morrey1952}. This condition is equivalent to the weak lower semicontinuity of the functional $I$. Should quasiconvexity fail, one can replace $f$ with its quasiconvex envelope (that is, the largest quasiconvex pointwise lower bound on $f$) to obtain a well-posed minimization problem with the same minimum value as the original one \cite{dacorogna08,roubicek20,Rindler2018}. However, quasiconvexity is an inherently non-local property in general \cite{Kristensen1999}, so it is usually difficult to verify. For this reason, sufficient conditions for the lower semicontinuity of the functional $I$ that are easier to manipulate remain of great interest.

In this work, we focus on the notion of \emph{polyconvexity}. A function $f$ of a matrix $X$ is polyconvex if it can be written as a convex function of $X$ and its minors. This local condition implies quasiconvexity when $f$ is finite (see, e.g., \cite[Theorem~5.3]{dacorogna08}) and ensures the weak lower semicontinuity of the functional $I(y)$ even if $f$ is not finite \cite{ball76}. Polyconvex functions are thus extremely popular in material modelling and nonlinear elasticity, where singularities arise naturally as a consequence of incompressibility constraints. They are also a key ingredient in relaxation-based numerical methods for integral minimization problems describing microstructures (see, e.g., \cite{Bartels2004,Bartels2006,bk11}) because polyconvex envelopes (rather than quasiconvex ones) can be approximated numerically~\cite{b05,ebg13,beg15,mppw24}.  Nevertheless, two fundamental difficulties remain.

First, verifying that a function $f(X)$ is polyconvex by finding an explicit convex representation in terms of the minors of $X$ requires considerable ingenuity. The same is true if one uses an alternative characterization of polyconvexity given in \cite[Theorem~5.6]{dacorogna08} (see also Proposition~\ref{prop:first-order-polyconvex} below), since this requires verifying that an inequality resembling the classical first-order characterization of convexity holds for all matrices $X,Y \in \R^{m\times n}$.

Second, existing numerical methods for computing polyconvex envelopes of non-polyconvex functions face scalability challenges. This is because they essentially grid the matrix space $\R^{m\times n}$ to implement Dacorogna's formula \cite[Theorem 6.8]{dacorogna08}, which expresses the polyconvex envelope at a matrix $X$ as a convex combination of function values at matrices $X_1,\ldots,X_N$ whose minors combine convexly to recover the minors of $X$. Finding these matrices and the convex combination weights is a nonconvex optimization problem that must be solved to \emph{global} optimality, which remains prohibitively expensive unless $m=n=2$ \cite{b05,ebg13,beg15} or the function being polyconvexified is isotropic \cite{mppw24}.

\subsection{Contributions and outline}
In this work, we develop systematic and computationally efficient methods to verify polyconvexity and compute polyconvex envelopes for functions $f:\R^{m\times n} \to \R \cup \{+\infty\}$ that are polynomial on their domain.

In Section~\ref{sec:polyconvexity} we leverage sum-of-squares (SOS) certificates of polynomial nonnegativity to define two computationally tractable strengthenings of polyconvexity for polynomials. The first, called \emph{lifted SOS polyconvexity}, enforces that $f$ be an SOS-convex polynomials of the minors of the input matrix. The second, called \emph{SOS polyconvexity}, replaces the inequality in the first-order characterization of polyconvexity from \cite[Theorem~5.6]{dacorogna08} with an SOS constraint. Both strengthenings are computationally tractable because SOS and SOS-convex polynomials can be recognized by solving semidefinite programs (SDPs), see for example \cite{Choietal1995,Lasserre2001,Parrilo2003,Nesterov2000,Laurent2009,Blekherman2013,lasserre15,nie23,Theobald2024}. We demonstrate their usefulness by giving new elementary proofs that the Alibert–Dacorogna–Marcellini function $\matrixnorm{X}^2 (\matrixnorm{X}^2 - 2 \det X)$ and the double-well function $\matrixnorm{X-I}^2 \matrixnorm{X+I}^2$ are polyconvex (Section~\ref{ss:sos-pc-examples}). Finally, in Section~\ref{ss:gaps} we investigate the gaps between polyconvexity, SOS polyconvexity and lifted SOS polyconvexity, giving examples of polynomials that are either polyconvex or SOS polyconvex but not lifted SOS polyconvex.

In Section~\ref{sec:envelope}, we exploit computational tools for polynomial optimization to develop an efficient algorithmic approach to computing polyconvex envelopes. The key observation, recalled in section~\ref{sec:linear}, is that Dacorogna's polyconvexification formula admits a reformulation as a \emph{linear} minimization problem over probability measures in $\R^{m\times n}$, whose dual is a maximization problem with a polynomial inequality constraint (see Section~\ref{sec:duality}). Section~\ref{sec:SDP-approx} then explains how tractable SDP approximations for this pair of infinite-dimensional problems can be formulated using the moment--SOS hierarchy from polynomial optimization. Each SDP in the hierarchy returns a lower bound on the value of the polyconvex envelope, and we discuss sufficient conditions to either guarantee convergence \emph{a priori} or verify it \emph{a posteriori}. Section \ref{sec:numerics} collects several numerical examples illustrating the efficiency of this approach, including a double-well function with dimensions $m=n=3$ whose polyconvex envelope at a given matrix $X$ can be evaluated  in a fraction of a second on a laptop with no need to exploit isotropy.

Finally, let us stress that the restriction to functions that are polynomial on their effective domains is not strong. Indeed, the techniques we describe can be generalized with straightforward modifications to the much broader class of \emph{semialgebraic functions}, that is, functions whose graph can be defined using a set of polynomial equations and inequalities. Examples include trigonometric functions and polynomials of matrix norms. We do not describe this generalization for brevity, but our methods remain broadly applicable.

\section{Sum-of-squares polyconvexity}\label{sec:polyconvexity}

In this section, we introduce and investigate computationally tractable sufficient conditions for the polyconvexity of polynomials, based on the notion of sum-of-squares (SOS) polynomials. We briefly review these notions next.

Throughout the paper, given positive integers $n,d$, we write $\R[x]_d$ for the space of polynomials of degree up to $d$ in the entries of $x \in \R^n$. This space has dimension $\smash{\binom{n+d}{d}}$. Inner products in an inner product space are always indicated using the angled brackets $\langle \cdot, \cdot \rangle$. In particular, given matrices $X$ and $Y$ of the same size, $\ip{X}{Y} := \tr(X^\top Y)$ is the Frobenius inner product and  $\matrixnorm{X}:=\sqrt{\ip{X}{X}}$ is the corresponding norm.

\subsection{SOS polynomials and SOS-convexity}

Let us start by recalling the definition of an SOS polynomial.

\begin{definition}[SOS polynomial]
	An even-degree polynomial $f \in \R[x]_{2d}$ is a \emph{sum of squares (SOS)} if there exist polynomials $g_1,\ldots,g_k \in \R[x]_d$ such that $f = g_1^2 + \cdots + g_k^2$.
\end{definition}
It is clear that SOS polynomials are globally nonnegative. While the converse is true only if $n=1$ or $d=2$ or $(n,d)=(2,4)$ \cite{Hilbert1888}, SOS polynomials are useful for practical computations involving nonnegative polynomials because they can be represented using positive semidefinite matrices. The link comes from the following statement, first observed in \cite{Choietal1995}. See \cite{Lasserre2001,Parrilo2003,Nesterov2000,Laurent2009,Blekherman2013,lasserre15} and more recently \cite{nie23,Theobald2024} and references therein.

\begin{proposition}
	Let $b:\R^n \to \smash{\R^{\binom{n+d}{d}}}$ be a polynomial vector whose entries span $\R[x]_d$. Then, $f \in \R[x]_{2d}$ is SOS if and only if there exists a positive semidefinite matrix $Q$ such that $f(x) = \ip{b(x)}{Q b(x)}$.
\end{proposition}

This result allows one to check if a given polynomial is SOS by solving a semidefinite program (SDP), a type of convex optimization problem for which mature algorithms and software exist. We refer interested readers to \cite{Wolkowicz2000,BenTal2001,Anjos2012} for more details.

SOS polynomials allow one to define a computationally tractable sufficient condition for the convexity of a polynomial, leading to the notion of SOS convexity. It was introduced originally in \cite{helton-nie10} to study semidefinite representations of convex sets.
SOS convexity is stronger than convexity, even though it is challenging to find polynomials that are convex but not SOS convex, see \cite{AhmadiParrilo2013} and \cite[Sec. 7.1.1]{nie23} for a recent survey.

\begin{definition}[SOS convexity]
	An $n$-variate polynomial $f(x)$ is \emph{SOS convex} if the $2n$-variate polynomial $\ip{y}{\nabla^2 f(x) y}$ is SOS in the variables $(x,y)$.
\end{definition}

An equivalent characterization of SOS convexity replaces the SOS condition involving the Hessian of $f$ with an SOS condition on its Bregman divergence,
\begin{equation}
	D_f(x,y)\ceq f(x)-f(y)- \ip{\nabla f(y)}{x-y}.
\end{equation}

\begin{proposition}[{\cite[Theorem 3.1]{AhmadiParrilo2013}}]\label{prop:sos-convexity}
	An $n$-variate polynomial $f$ is SOS convex if and only if its Bregman divergence $D_f(x,y)$ is SOS in the variables $(x,y)$.
\end{proposition}

Finally, we mention that SOS-convex polynomials are a strict subset of convex polynomials in general. The following statement completely characterizes these gaps.
\begin{theorem}[{\cite[Theorem~5.1]{AhmadiParrilo2013}}]\label{thm:AP-convex-not-sosconvex}
	There exist $n$-variate polynomials of degree $d$ that are convex but not SOS convex unless $n=1$ or $d=2$ or $(n,d)=(2,4)$.
\end{theorem}

\subsection{SOS polyconvexity}
\label{sec:SOS-polyconvexity}
Fix positive integers $m,n \in \N$, write $m \wedge n = \min\{m,n\}$ and set $N= \sum_{s=0}^{m \wedge n} \binom{m}{s}\binom{n}{s}$. Let
$p:\R^{m\times n}\to \R^N$
denote the function mapping an $m\times n$ matrix $X$ to the collection of its minors. (The precise order in which these minors are listed is not important.) We recall the definition of a polyconvex function.

\begin{definition}[Polyconvexity]\label{def:polyconvex}
	A function $f:\R^{m\times n}\to\R \cup \{+\infty\}$ is \emph{polyconvex} if there exists a convex function
	$g:\R^N\to\R \cup \{+\infty\}$ such that $f = g \circ p$.
\end{definition}

We also recall the following equivalent characterization of polyconvexity, which resembles the usual first-order characterization of convexity.

\begin{proposition}[{\cite[Theorem~5.6]{dacorogna08}}]\label{prop:first-order-polyconvex}
	A function $f:\R^{m\times n} \to \R$ is polyconvex if and only if there exists a function
	$q:\R^{m\times n}\to \R^N$ such that
	\begin{equation}\label{eq:first-order-polyconvex}
		f(X)-f(Y)- \ip{q(Y)}{p(X)-p(Y)} \geq 0 \quad \forall X,Y \in \R^{m\times n}.
	\end{equation}
\end{proposition}

Based on these two characterizations, we introduce two computationally tractable sufficient conditions for the polyconvexity of a polynomial. In the first and perhaps most natural one, we insist that the function $g$ appearing in Definition \ref{def:polyconvex} be an SOS-convex polynomial. In the second, we restrict the map $q$ in Proposition \ref{prop:first-order-polyconvex} to be a polynomial and strengthen inequality \cref{eq:first-order-polyconvex} into an SOS constraint.

\begin{definition}[Lifted SOS polyconvexity]\label{def:lifted-sos-pc}
	A polynomial $f:\R^{m\times n}\to\R$ is \emph{lifted SOS polyconvex} if there exists an SOS-convex polynomial $g:\R^N\to\R$ such that $f = g \circ p$.
\end{definition}	

\begin{definition}[SOS polyconvexity]\label{def:first-order-sos-pc}
	A polynomial $f:\R^{m\times n}\to\R$ is \emph{SOS polyconvex} if there exists a \emph{polynomial} map
	$q:\R^{m\times n}\to\R^N$ such that
	$f(X)-f(Y)- \ip{q(Y)}{p(X)-p(Y)}$
	is a sum of squares polynomial in the entries of $X$ and $Y$.
\end{definition}

Note that these definitions allow for polynomials $g$ and $q$ of arbitrarily large degree. For practical computations that translate SOS constraints into SDPs, these degrees must be fixed a priori. For a given degree choice, any point in the feasible set of the SDP certifies polyconvexity. Infeasibility of the SDP, instead, is inconclusive and one should repeat the test for larger polynomial degrees.

\subsection{Examples}\label{ss:sos-pc-examples}
We now demonstrate the usefulness of SOS polyconvexity and lifted SOS polyconvexity by providing elementary proofs that two classical functions are polyconvex.

\subsubsection{The Alibert--Dacorogna--Marcellini function}\label{ss:dam-sos-pc}

It is well known (see \cite{DacorognaMarcellini1988,AlibertDacorogna1992} or \cite[\S5.3.8]{dacorogna08}) that the function
$f(X) = \matrixnorm{X}^2(\matrixnorm{X}^2 - 2\det X)$ 
is polyconvex. Here, we prove that it is in fact SOS polyconvex. Specifically, fixing the order of the minors vector to be $p(X)=(X_{11}, X_{21}, X_{12}, X_{22}, \det X)$, we use semidefinite programming to search for a cubic polynomial vector $q:\R^{2\times 2} \to \R^5$ such that the polynomial
$f(X)-f(Y)-\ip{q(Y)}{p(X)-p(Y)}$
is SOS. After rounding numerical coefficients, the SDP solution gives
\begin{equation}
	q(Y) = 
	\begin{pmatrix}
		4Y_{11}(\matrixnorm{Y}^2 - \det Y) \\
		4Y_{21}(\matrixnorm{Y}^2 - \det Y) \\
		4Y_{12}(\matrixnorm{Y}^2 - \det Y) \\
		4Y_{22}(\matrixnorm{Y}^2 - \det Y) \\
		-2\matrixnorm{Y}^2
	\end{pmatrix}.
\end{equation}
Interestingly, this is precisely the choice made in \cite[\S5.3.8]{dacorogna08}. 

The success of our SDP computation provides a numerical certificate that $f(X)$ is SOS polyconvex, but this certificate is subject to roundoff error. For an analytical proof, we set
\begin{equation*}
	S(X) := \begin{pmatrix}
		X_{22} & -X_{21} \\ -X_{12} & X_{11}
	\end{pmatrix}
\end{equation*}
and verify through a lengthy but straightforward calculation that the polynomial $f(X)-f(Y)-\ip{q(Y)}{p(X)-p(Y)}$ admits the following SOS decomposition:
\begin{align}
	&\frac14 \left( \matrixnorm{X}^2 - 2\det X - \ip{X}{Y} + \ip{X}{S(Y)}\right)^2
	\nonumber \\
	+
	&\frac14 \left( \matrixnorm{Y}^2 - 2\det Y - \ip{Y}{X} + \ip{Y}{S(X)}\right)^2
	\nonumber \\
	+
	&\frac14 \left( 2\det X - 2\ip{X}{Y} - 2\det Y - \matrixnorm{X}^2 + 3\matrixnorm{Y}^2 \right)^2
	\nonumber \\
	+
	&\frac34 \Big( (X_{12} + X_{21})(Y_{11} - Y_{22}) + (X_{22} - X_{11})(Y_{12} + Y_{21}) \Big)^2
	\nonumber \\
	+
	&\frac14 \Big( (X_{11} + X_{22})(Y_{12} - Y_{21}) + (X_{21} - X_{12})(Y_{11} + Y_{22})\Big)^2
	\nonumber \\
	+
	&\frac14 \Big( (X_{12} - X_{21})(X_{11} - X_{22}) - (Y_{12} - Y_{21})(Y_{11} - Y_{22}) \Big)^2
	\nonumber \\
	+
	&\frac14 \Big( (X_{12} + X_{21})(X_{11} + X_{22}) - (Y_{12} + Y_{21})(Y_{11} + Y_{22}) \Big)^2
	\nonumber \\
	+
	&\frac14 \Big( (X_{11} - X_{22})^2  + (X_{12} + X_{21})^2 \Big) (X_{12} - X_{21} - Y_{12} + Y_{21})^2
	\nonumber \\
	+
	&\frac14 \Big( (X_{11} - X_{22})^2 +  (X_{12} + X_{21})^2 \Big) (X_{11} + X_{22} - Y_{11} - Y_{22})^2
	\nonumber \\
	+
	&\frac14 \Big( (Y_{11} - Y_{22})^2 + (Y_{12} + Y_{21})^2 \Big) (Y_{12} - Y_{21} - X_{12} + X_{21})^2
	\nonumber \\
	+
	&\frac14 \Big( (Y_{11} - Y_{22})^2 +  (Y_{12} + Y_{21})^2 \Big) (Y_{11} + Y_{22} - X_{11} - X_{22})^2
	\nonumber \\
	+
	&\frac14 \Big( X_{11}^2 - X_{22}^2 - Y_{11}^2 + Y_{22}^2 \Big)^2
	+\frac14 \Big( X_{12}^2 - X_{21}^2 - Y_{12}^2 + Y_{21}^2 \Big)^2. \nonumber
\end{align}
%
\subsubsection{A double-well function}
The double-well function 
\begin{equation*}
	f(X)=\matrixnorm{X-I}^2\matrixnorm{X+I}^2
\end{equation*}
was recently shown to be polyconvex \cite{hk25}. Here, we prove that it can be written as an SOS-convex polynomial of the minors of $X$, so it is in fact lifted SOS polyconvex. 

To accomplish this, we begin by expanding $\matrixnorm{X\pm I}^2 = \matrixnorm{X}^2 + n \pm 2 \ip{X}{I}$ to rewrite
\begin{align*}
	f(X) 
	&= \left( \matrixnorm{X}^2 + n \right)^2 - 4\ip{X}{I}^2 \\
	&= \matrixnorm{X}^4 + 2n \matrixnorm{X}^2 +n^2 - 4\tr( X )^2.
\end{align*}
Next, we observe that
\begin{align*}
	4\tr( X )^2 
	&= 4\tr( X^2 ) + 4\tr( X )^2 - 4\tr( X^2 ) \phantom{\sum_{i<i}}\\
	&= 4\langle{X^\top},{X}\rangle + 8\sum_{i<j} \left( X_{ii}X_{jj} - X_{ij}X_{ji} \right)\\
	&= \matrixnorm{X+X^\top}^2 - \matrixnorm{X-X^\top}^2 + 8\sum_{i<j} \left( X_{ii}X_{jj} - X_{ij}X_{ji} \right).
\end{align*}
Using this identity, together with the equality $\matrixnorm{X}^2 = \frac14\matrixnorm{X+X^\top}^2+\frac14\matrixnorm{X-X^\top}^2$, we may rewrite
\begin{align}\label{e:dw-sos-convex-repr}
	f(X) 
	= \matrixnorm{X}^4 + \frac12 (n-2) \matrixnorm{X+X^\top}^2 + \frac12 (n+2) \matrixnorm{X-X^\top}^2&
	\nonumber \\
	+ 8\sum_{i<j} \left( X_{ij}X_{ji} - X_{ii}X_{jj} \right) + n^2&.
\end{align}
We claim that the expression on the right-hand side is a positive linear combination of SOS-convex polynomials of the minors of $X$. This implies that the whole expression is an SOS-convex polynomial of the minor of $X$, proving that $f(X)$ is lifted SOS polyconvex.

To prove our claim, observe that $n^2$ is a constant, hence SOS-convex. Each term $X_{ij}X_{ji} - X_{ii}X_{jj}$ is a linear (hence, SOS-convex) polynomial of the minors of $X$ since it is precisely one of its $2 \times 2$ minors. The functions $\smash{\matrixnorm{X + X^\top}^2}$ and $\smash{\matrixnorm{X - X^\top}^2}$ are convex quadratic polynomials of the entries of $X$, so they are SOS-convex by Theorem \ref{thm:AP-convex-not-sosconvex}. Finally, the function $\matrixnorm{X}^4$ is SOS-convex because its Hessian $H(X)$, viewed as a linear operator on $\R^{n\times n}$, satisfies $\langle Z, H(X)Z \rangle = 4 \ip{X}{Z}^2 + 2 \matrixnorm{X}^2\matrixnorm{Z}^2$ and the right-hand side is clearly a sum of squares.

\subsection{Gaps between polyconvexity and SOS polyconvexity}
\label{ss:gaps}
We now turn to studying the relation between polyconvexity, SOS polyconvexity, and lifted SOS polyconvexity. Our first result shows that these notions are nested.

\begin{theorem}
	Let $f:\R^{m\times n}\to\R$ be polynomial. Then,
	\begin{equation*}
		f \text{ is lifted SOS polyconvex}
		\quad\implies\quad
		f \text{ is SOS polyconvex}
		\quad\implies\quad
		f \text{ is polyconvex}.
	\end{equation*}
\end{theorem}

\begin{proof}
	We only need to prove the first implication, since the second one is clear.
	Let $g:\R^N \to \R$ be the SOS-convex polynomial such that $f= g\circ p$.  By Proposition~\ref{prop:sos-convexity},
	its Bregman divergence $D_g(q,r)$ is an SOS polynomial in the entries of $q,r \in \R^N$. Then, since the vectors $p(X),p(Y) \in \R^N$ depend polynomially on $X$ and $Y$, the function
	\begin{equation*}
		(X,Y) \mapsto f(X)-f(Y)- \ip{\nabla g(p(Y))}{p(X)-p(Y)}  = D_g(p(X),p(Y))
	\end{equation*}
	is an SOS polynomial in the entries of $X$ and $Y$. This means that $f$ satisfies the definition of SOS polyconvexity with $q(Y)=\nabla g(p(Y))$.
\end{proof}

The remaining results in this section show that the reverse implications are \emph{false} in general. We start by building on the known gap between convexity and SOS-convexity (cf. Theorem~\ref{thm:AP-convex-not-sosconvex}) to exhibit a gap between SOS polyconvexity and polyconvexity. We focus on the case $m,n\geq 2$ because, when $m=1$ or $n=1$, polyconvexity and SOS polyconvexity reduce to convexity and SOS-convexity, respectively.

\begin{theorem}\label{thm:pc-not-SOS-pc}
	Fix $m,n,d \in \N$ such that
	$m, n\geq 2$, $d>2$, and $(m, d) \neq (2,4)$ or $(n,d) \neq (2,4)$.
	There exists a polyconvex polynomial $f:\R^{m\times n} \to \R$ of degree $d$ that is not SOS polyconvex.
\end{theorem}

\begin{remark}
	The cases $d=2$ or $(m,d)=(n,d)=(2,4)$ remain open. Convexity and SOS convexity coincide in these cases, but this does not rule out a gap between polyconvexity and SOS polyconvexity.
\end{remark}

\begin{proof}
	We assume for definiteness that $(m,d)\neq (2,4)$; the case $(n,d)\neq (2,4)$ is handled similarly. Since $m>1$, $d>2$ and $(m,d) \neq (2,4)$, we can use Theorem \ref{thm:AP-convex-not-sosconvex} to choose an $m$-variate convex polynomial $\hat{f}$ of degree $d$ that is not SOS convex. We claim that the function
	\begin{equation*}
		\begin{array}{llll}
			f:&\R^{m\times n} & \to & \R \\
			&X& \mapsto & f(X)=\hat{f}(X_{11},\ldots,X_{m1})
		\end{array}
	\end{equation*}
	is polyconvex but not SOS polyconvex.
	
	Polyconvexity is clear. We now prove that $f$ is not SOS polyconvex because SOS polyconvexity would imply that $\hat{f}$ is SOS convex. Precisely, suppose there exists a polynomial map $q:\R^{m\times n} \to \R^N$ such that
	\begin{equation*}
		f(X) - f(Y) - \ip{q(Y)}{p(X)-p(Y)}  \text{ is SOS}.
	\end{equation*}
	Now, let $e_1 \in \R^n$ be the first canonical unit vector and the matrix-valued linear function $F:\R^{m} \to \R^{m\times n}$ defined by $F(x) = x\otimes e_1$. Since the only nonzero minors of $F(x)$ are those depending linearly on $x$, there exists a polynomial map $\hat{q}:\R^m \mapsto \R^m$ such that
	\begin{equation*}
		\ip{q(F(y))}{p(F(x))} = 
		\ip{\hat{q}(y)}{x} \quad \forall x,y\in\R^m.
	\end{equation*}
	Then, the polynomial
	\[
	\hat{f}(x) - \hat{f}(y) - \ip{\hat{q}(y)}{x-y}
	= f(F(y)) - f(F(x)) - \ip{q(F(x))}{p(F(y))-p(F(x))}
	\]
	is a sum of squares. This, in particular, means that $\hat{q}(y)$ is a subgradient of the polynomial $\hat{f}$ at $y$, and since $\hat{f}$ is a polynomial we must in fact have $\hat{q}(y) = \nabla \hat{f}(y)$. We then conclude that $\hat{f}(x) - \hat{f}(y) - \ip{\nabla \hat{q}(y)}{x-y}$ is an SOS polynomial, so $\hat{f}$ is SOS-convex by Proposition~\ref{prop:sos-convexity}.
\end{proof}

Next, we prove that the set of lifted SOS polyconvex polynomials is strictly smaller than the set of SOS polyconvex polynomials in general. We start with an explicit example based on the same polyconvex function studied in Section~\ref{ss:dam-sos-pc} and in \cite{DacorognaMarcellini1988,AlibertDacorogna1992}.

\begin{theorem}\label{thm:dam}
	The function $f:\R^{2\times 2} \to \R$ defined by $f(X)=\matrixnorm{X}^2 ( \matrixnorm{X}^2 - 2\det X )$ is SOS polyconvex but not lifted SOS polyconvex.
\end{theorem}

We postpone the proof of this result to Section~\ref{ss:proof-DAM} and, instead, immediately use it to deduce the following general statement.
\begin{theorem}
	If $m,n \in \N$ satisfy $m, n\geq 2$, there exists a degree-four SOS polyconvex polynomial $f:\R^{m\times n} \to \R$ that is not lifted SOS polyconvex.
\end{theorem}

\begin{proof}
	Let $B:\R^{m\times n} \to \R^{2\times 2}$ be the linear map defined via
	\begin{equation*}
		B(X) = \begin{pmatrix}
			X_{11} & X_{12} \\ X_{21} & X_{22}
		\end{pmatrix}.
	\end{equation*}
	Let $f:\R^{2\times 2} \to \R$ be the function from Theorem~\ref{thm:dam}. We claim that the function
	\begin{equation*}
		\begin{array}{llll}
			F:&\R^{m\times n} & \to & \R \\
			&X& \mapsto & f(B(X))
		\end{array}
	\end{equation*}
	is SOS polyconvex but not lifted SOS polyconvex. We prove these two claims separately.
	
	\emph{Claim 1: $F$ is SOS polyconvex.} The function $f$ is SOS polyconvex by Theorem \ref{thm:dam}, so there exists a polynomial map $q:\R^{2\times 2} \to \R^5$ such that the polynomial $f(Y)-f(Z)- \ip{q(Z)}{Y - Z}$ is SOS in the entries of the matrices $Z,Y\in\R^{2\times 2}$. Now, since the five minors listed in $p(B(X)) \in \R^5$ are also listed in $p(X) \in\R^N$, there exists a linear map $A:\R^N \to \R^5$ such that $A(p(X)) = p(B(X))$. Define a polynomial map $Q:\R^{m\times n} \to \R^N$ via $Q(Y) = A^*(q(B(Y)))$, where $A^*$ is the adjoint of $A$. The polynomial
	\begin{align*}
		F(X) - F(Y) - \ip{Q(Y)}{p(X) - p(Y)}
		&= f(B(X)) - f(B(Y) ) - \ip{q(B(Y))}{p(B(X)) - p(B(Y))}
	\end{align*}
	is SOS in $X$ and $Y$, so $F$ is SOS polyconvex.

	\emph{Claim 2: $F$ is not lifted SOS polyconvex.} We prove that if $F$ is lifted SOS polyconvexty, then so is $f$, which is false by Theorem~\ref{thm:dam}. Let $C:\R^{2\times 2} \to \R^{m\times n}$ be the linear map that places $Y \in \R^{2\times 2}$ as the $2\times 2$ top-left block of an otherwise zero $m\times n$ matrix. Since the only nonzero minors of the matrix $C(Y)$ are the minors of $Y$, there exists a linear map $L:\R^5 \to \R^N$ such that $L(p(Y)) = p(C(Y))$. Now, suppose $F$ if lifted SOS polyconvex. Then, there exists an SOS-convex polynomial $G:\R^N \to \R$ such that $F(X)=G(p(X))$. The polynomial $g=G\circ L$ is SOS-convex because it is the composition of an SOS-convex polynomial and a linear function. But then, for all matrices $Y\in\R^{2\times 2}$ we have that
	$f(Y) = F(C(Y)) = G(p(C(Y))) = G(L(p(Y))) = g(p(Y))$, so $f$ is lifted SOS polyconvex.
\end{proof}

\subsubsection{Proof of Theorem \ref{thm:dam}}
\label{ss:proof-DAM}

We now prove Theorem \ref{thm:dam}. Since we have already demonstrated in Section \ref{ss:dam-sos-pc} that the function $f(X)=\matrixnorm{X}^2(\matrixnorm{X}^2 - 2\det X)$ is SOS polyconvex, we only need to show that it is not lifted SOS-polyconvex. This is a consequence of the following stronger result.

\begin{proposition}
	Let $f:\R^{2\times 2} \to \R$ be given by $f(X)=\matrixnorm{X}^2 ( \matrixnorm{X}^2 - 2\det X )$. If a function $g:\R^{2\times 2} \times\R \to \R$ is convex and $g(X, \det X)=f(X)$, then $g$ is not polynomial. 
\end{proposition}

\begin{remark}
	An explicit, piecewise-polynomial convex function $g$ satisfying $g(X, \det X)=f(X)$ has been given in \cite{Hartwig1995} (see also \cite[Sect.~5.3.8]{dacorogna08}).
\end{remark}

\begin{proof}
	We assume there exists a convex polynomial $g$ such that $f(X) = g(X, \det X)$ and derive a contradiction in three steps. First, we show that $g$ must be parametrized by a set of polynomials $h_0,\ldots, h_p$ with a special structure. Second, we explicitly determine the polynomials $h_0$, $h_1$ and $h_2$ up to a single nonnegative constant. Finally, we use this explicit characterization to show that $\nabla^2 g$ is not positive semidefinite, contradicting the assumed convexity of $g$.
	
	\paragraph{Step 1: Parameterizing $g$.}
	By construction, the polynomial $g(X,y) - f(X)$ vanishes when $y-\det X=0$. It is not difficult to show (for instance, using \cite[Lemma~2.2]{Laurent2009}) that the ideal generated by the polynomial $p(X,y) = y - \det X$ is real radical. Then, the real Nullstellensatz (see, e.g., \cite[Theorem~2.1]{Laurent2009}) implies that $g(X,y)-f(X)$ must belong to this ideal. This means there exists a polynomial $h$ such that
	\begin{equation}
		g(X,y) = f(X) + (y - \det X) h(X,y).
	\end{equation}
	We now express $h(X, y)$ in a convenient way that reflects the nature of $y$ as a `dummy variable' for $\det X$, which is quadratic in the entries of $X$. Specifically, there exists $p \in \mathbb{N}$ such that we can write
	\begin{subequations}
		\begin{equation}
			h(X, y) = \sum\limits_{d=0}^{p} h_d(X, y)
		\end{equation}
		with
		\begin{equation}
			h_d(X, y) 
			\in \linspan\left\{
			X_{11}^{\alpha_1} X_{12}^{\alpha_2} X_{21}^{\alpha_3} X_{22}^{\alpha_4} y^{\alpha_5}:\;
			\alpha_1 + \alpha_2 + \alpha_3 +\alpha_4 + 2 \alpha_5 = d
			\right\}.
		\end{equation}
	\end{subequations}
	In particular,
	\begin{subequations}\label{e:h-polys}
		\begin{align}
			h_0(X,y) 
			= 
			&\;c_{0,1}
			\\
			h_1(X,y)
			=
			&\;c_{1,1} X_{11} +
			c_{1,2} X_{21} +
			c_{1,3} X_{12} +
			c_{1,4} X_{22}
			\\
			h_2(X,y)
			=
			&\;c_{2,1} y +
			c_{2,2} X_{11}^2 +
			c_{2,3} X_{11} X_{21} +
			c_{2,4} X_{11} X_{12} +
			c_{2,5} X_{11} X_{22} +
			c_{2,6} X_{21}^2
			\nonumber \\
			&+
			c_{2,7} X_{12} X_{21} +
			c_{2,8} X_{21} X_{22} +
			c_{2,9} X_{12}^2 +
			c_{2,10} X_{12} X_{22} +
			c_{2,11} X_{22}^2
		\end{align}
	\end{subequations}
	for some real coefficients $c_{i,j}$. Next, we determine all but one of these coefficients using the assumption that $g$ be convex.
	
	\paragraph{Step 2: Determining $h_0$, $h_1$, $h_2$.}
	We compute the hessian $\nabla^2 g$ and evaluate it at $X=0$ to find that
	\begin{equation*}
		\nabla^2 g(0,y) = \sum_{k=0}^p A_k y^k
	\end{equation*}
	for suitable $5 \times 5$ coefficient matrices $A_k$. In particular, the matrices $A_0$ and $A_1$ can be computed explicitly because they are fully determined by the polynomials $h_0,\ldots,h_4$, which contain a finite (albeit large) number of terms. One finds that
	\begin{equation*}
		A_0 = 
		\begin{pmatrix}
			0 & 0 & 0 & -c_{0,1} & c_{1,1} \\
			0 & 0 & c_{0,1} & 0 & c_{1,2} \\
			0 & c_{0,1} & 0 & 0 & c_{1,3} \\
			-c_{0,1} & 0 & 0 & 0 & c_{1,4} \\
			c_{1,1} & c_{1,2} & c_{1,3} & c_{1,4} & 2 c_{2,1} \\
		\end{pmatrix}
	\end{equation*}
	and
	\begin{equation*}
		A_1 = 
		\begin{pmatrix}
			2c_{2,2} & c_{2,3} & c_{2,4} & c_{2,5} - c_{2,1} & \ast \\
			c_{2,3} & 2c_{2,6} & c_{2,1} + c_{2,7} & c_{2,8} & \ast \\
			c_{2,4} & c_{2,1} + c_{2,7} & 2c_{2,9} & c_{2,10} & \ast \\
			c_{2,5} - c_{2,1} & c_{2,8} & c_{2,10} & 2c_{2,11} & \ast \\
			\ast & \ast & \ast & \ast & \ast
		\end{pmatrix},
	\end{equation*}
	where $\ast$ denotes entries that are not used below.

	Now, since $g$ is assumed to be convex, we must have $\nabla^2 g(0,y) \succeq 0$ for all $y$. For $y=0$ we obtain $\nabla^2 g(0,0) = A_0 \succeq 0$, which requires
	\begin{equation}\label{e:h-coeff-1}
		c_{0,1}=c_{1,1}=\cdots = c_{1,4}=0 \quad\text{and}\quad c_{2,1}\geq 0.
	\end{equation}
	In particular, the $4\times 4$ top-left block of $A_0$ must vanish. But then, the same block of the positive semidefinite matrix polynomial $\nabla^2 g(0,y)$ cannot have linear terms in $y$, so the $4\times 4$ top-left block of $A_1$ must vanish. This gives
	\begin{subequations}
		\label{e:h-coeff-2}
		\begin{gather}
			c_{2,2}=c_{2,3}=c_{2,4}=c_{2,6}=c_{2,8}=c_{2,9}=c_{2,10}=c_{2,11}=0,\\
			c_{2,5}=c_{2,1},\\
			c_{2,7}=-c_{2,1}.
		\end{gather}
	\end{subequations}
	Substituting \cref{e:h-coeff-1,e:h-coeff-2} into \cref{e:h-polys} yields
	\begin{equation}
		h_0(X,y) 
		= 
		0,
		\quad
		h_1(X,y)
		=0,
		\quad\text{and}\quad
		h_2(X,y)
		=
		c_{2,1} \left( y + \det X \right).
	\end{equation}

	\paragraph{Step 3: Deriving a contradiction.}
	We now recompute $\nabla^2 g(X,y)$ using $h_0=h_1=0$ and $h_2(X,y) = c_{2,1} \left( y + \det X \right)$, and evaluate the resulting polynomial matrix at
	\begin{equation*}
		X = \begin{pmatrix} z & 0 \\ 0 & 2z \end{pmatrix}
		\quad\text{and}\quad
		y=0.
	\end{equation*}
	This results in a polynomial matrix with indeterminate $z$, whose $4 \times 4$ top-left block has the form
	$H(z) = z^2 B + \text{higher degree terms}$. The matrix $B$ is
	\begin{equation*}
		B = \begin{pmatrix}
			4 - 8c_{2,1} & 0 & 0 & - 8c_{2,1} - 14 \\
			0 & 12 & 4c_{2,1} + 10 & 0 \\
			0 & 4c_{2,1} + 10 & 12 & 0 \\
			- 8c_{2,1} - 14 & 0 & 0 & 28 - 2c_{2,1}
		\end{pmatrix} .
	\end{equation*}
	Now, since $g$ is SOS-convex by assumption, the polynomial matrix $H(z)$ must be positive semidefinite. This requires $B\succeq 0$, which in turn requires the nonnegativity of the minor
	\begin{equation*}
		\left( 4-8c_{2,1} \right)\left( 28-2c_{2,1} \right) - \left( 8c_{2,1}+14 \right)^2
		= - 48 c_{2,1}^2 - 456 c_{2,1} - 84.
	\end{equation*}
	Since $c_{2,1}\geq 0$ by \cref{e:h-coeff-1}, however, this minor is strictly negative.
	This is the contradiction we need to conclude that the function $f(X)=|X|^2(|X|^2-2\det X)$ is not lifted SOS-polyconvex.
\end{proof}

\section{Computing the polyconvex envelope}\label{sec:envelope}

We now turn our attention to computing the polyconvex envelope of non-polyconvex functions that are polynomial on their effective domain. 
The polyconvex envelope of $f$, denoted by $\fpc:\R^{m\times n}\to\R\cup\{\pm\infty\}$, 
is the pointwise supremum of all polyconvex lower bounds of $f$.
That is, for every $X\in\R^{m\times n}$ we 
have $\fpc(X):=\sup_h h(X)$ where the supremum is over all polyconvex 
functions $h:\R^{m\times n} \to \R\cup \{+\infty\}$ such that $h(X)\leq f(X)$ for all  $X\in\R^{m\times n}$.

\subsection{Polyconvexification as a linear minimization over measures}\label{sec:linear}
We begin by recalling that if $f:\R^{m\times n}\to\R\cup\{\pm\infty\}$ is lower semicontinuous and bounded below, then its polyconvex envelope can be evaluated at any fixed matrix $X \in \R^{m\times n}$ by solving a linear minimization problem over a particular subset of probability measures on $\R^{m\times n}$.

In what follows, $\Pr(\R^{m\times n})$ denotes the set of Borel probability measures on $\R^{m \times n}$, meaning Borel measures on $\R^{m \times n}$ such that $\int d\mu= 1$. Recall that a Borel measure on $\R^{m\times n}$ is a nonnegative and countably additive function on the Borel $\sigma$-algebra of $\R^{m\times n}$, that is, the smallest collection of subsets of the space that contains the open sets and is closed under countable unions and complementation.

\begin{theorem}\label{th:pc-measure-lp}
	Suppose $f:\R^{m\times n} \to \R\cup\{+\infty\}$ is bounded below. Then, for every $X\in\R^{m\times n}$,
	\begin{equation}\label{primal}
		\fpc(X) = \inf_{\substack{\mu \in \Pr(\R^{m \times n}) \\ \int p\, d\mu = p(X)}} \int f\, d\mu.
	\end{equation}   
\end{theorem}

\begin{remark}
	This result restates Dacorogna's formula for the polyconvex envelope~\cite[Theorem~5.6]{dacorogna08} using general probability measures instead  of atomic ones. It is not new: versions of problem \cref{primal} have appeared, for example, in \cite[Eq.~(6.2)]{Firoozye1991} and in~\cite[Eq.~(35)]{bk11}.
\end{remark}

\begin{proof}
	The proof follows the same arguments in ~\cite[Sect.~5.2.3]{dacorogna08}. Let $\phi(X)$ denote the right-hand side of \eqref{primal}. We need to show that $\phi$ is the largest polyconvex lower bound on on $f$. 
	The inequality $\phi(X) \leq f(X)$ for every matrix $X$ follows because the atomic measure $\mu = \delta_X$ is feasible for the minimization defining $\phi(X)$. (Note that the both sides may be infinite.)
	
	To show that $\phi$ is polyconvex, note that $\phi(X)=\psi(p(X))$ where $\psi:\R^N \to \R\cup\{+\infty\}$ is given by
	\begin{equation*}
		\psi(q) := \inf_{\substack{\mu \in \Pr(\R^{m \times n}) \\ \int p\, d\mu = q}} \int f\, d\mu.
	\end{equation*}
	This function is convex, which implies the polyconvexity of $\phi$. To see this, we fix arbitrary $q_1,q_2 \in \R^N$ and $\lambda \in [0,1]$ and verify that
	\begin{equation}\label{e:jensen}
		\psi(\lambda q_1 + (1-\lambda)q_2) \leq \lambda \psi(q_1) + (1-\lambda) \psi(q_2).
	\end{equation}
	This inequality holds trivially if either term on the right-hand side is infinite. If they are finite, then for any $\varepsilon>0$ there exist probability measures $\mu_1$ and $\mu_2$ satisfying
	$\int f d\mu_i \leq \psi(q_i) + \varepsilon$ for $i\in\{1,2\}$.
	Since the convex combination $\lambda\mu_1 + (1-\lambda)\mu_2$ is feasible for the minimization defining $\psi(\lambda q_1 + (1-\lambda)q_2)$, we find that
	$\psi(\lambda q_1 + (1-\lambda)q_2) \leq \lambda \psi(q_1) + (1-\lambda) \psi(q_2) + \varepsilon$. Inequality \eqref{e:jensen} follows because $\varepsilon$ is arbitrary. 

	Finally, we show that if $h$ is a polyconvex lower bound for $f$, then $h\leq \phi$. For this, write $h = \theta \circ p$ for a convex function $\theta$. The inequality $h\leq f$ implies that $\int f\, d\mu \geq \int \theta \circ p\, d\mu$ for every probability measure $\mu$, so for every matrix $X$ we find that
	\begin{equation*}
		\phi(X)
		\geq 
		\inf_{\substack{\mu \in \Pr(\R^{m \times n}) \\ \int p\, d\mu = p(X)}} \int \theta\circ p \, d\mu
		\geq \inf_{\substack{\mu \in \Pr(\R^{m \times n}) \\ \int p\, d\mu = p(X)}} \theta\left( \int p\, d\mu \right) 
		= \theta(p(X)) 
		= h(X).
		\qedhere
	\end{equation*}
\end{proof}

\subsection{Duality and lower semicontinuity of the polyconvex envelope}\label{sec:duality}

Next, we show that if $f$ satisfies a suitable coercivity condition, then its polyconvex envelope can be evaluated by solving a linear problem over nonnegative functions dual to the measure-theoretic minimization from Theorem~\ref{th:pc-measure-lp}.

\begin{theorem}\label{th:duality-improved}
	Let $f:\R^{m \times n} \to \R \cup \{+\infty\}$ be a lower semicontinuous function 
	such that $f(X_k)/|X_k|^{\min(m,n)}\to +\infty$ whenever $|X_k|\rightarrow \infty$.
	Then for every $X\in\dom f$,
	\begin{equation}\label{dual}
		\fpc(X) = \sup_{(\dualu,\dualv)\in\R^N\times\R}
		\dualv + \langle \dualu, p(X) \rangle
		\quad\text{s.t.}\quad 
		f(Z) - \dualv - \langle \dualu, p(Z) \rangle\geq 0 
		\quad \forall Z \in \dom f.
	\end{equation}
\end{theorem}

\begin{remark}
	The optimal value of the dual maximization problem in \cref{dual} is a lower bound on $\fpc(X)$ for every $f$.
	This is because the function $Z \mapsto \dualv + \langle \dualu, p(Z) \rangle$ is a polyconvex lower bound on $f$ for every feasible choice of $\dualu\in\R^N$ and $\dualv\in\R$. The continuity and coercivity assumptions on $f$ in Theorem~\ref{th:duality-improved} guarantee that the lower bound is sharp.
\end{remark}

\begin{proof}
	Since the entries of $p(X)$ are minors of $X$, they are polynomial of degree at most $\min(m,n)$. Then, 
	for every $\dualu\in \R^N$, the lower semicontinuity and coercivity assumptions on $f$ ensure that the 
	function
	\begin{equation*}
		\begin{array}{llll}
			\varphi_{\dualu}:&\R^{m\times n} & \to & \R \\
			&Z& \mapsto & f(Z) - \langle \dualu, p(Z)\rangle
		\end{array}
	\end{equation*}
	has compact sublevel sets. In particular, $\varphi_{\dualu}$ is bounded below and attains 
	a minimum. It follows that $\varphi_{\dualu}$ is $\mu$-integrable for every probability 
	measure $\mu$ in the set
	\[          
	\mathcal{X}:=\left\{\mu\in\Pr(\R^{m \times n}) 
	\colon 
	\supp\mu\subset \dom f \text{ and } f \text{ is $\mu$-integrable}
	\right\}.
	\]
	Without loss of generality, we assume that $f$ is nonnegative, as the statement of
	Theorem~\ref{th:duality-improved} is invariant under adding finite constants to $f$.
	\par
	We claim the following chain of equalities:
	\begin{subequations}
		\begin{align}
			\fpc(X) 
			&= 
			\inf_{\substack{\mu \in \mathcal{X} \\ \int p \, d\mu = p(X)}} 
			\int f \, d\mu
			\label{e:minimax-1}
			\\
			&= 
			\adjustlimits
			\inf_{\mu \in \mathcal{X}} 
			\sup_{\dualu \in \R^N}
			\int f - \langle \dualu, p \rangle \, d\mu + \langle \dualu, p(X) \rangle
			\label{e:minimax-2}
			\\
			&= 
			\adjustlimits
			\sup_{\dualu \in \R^N}
			\inf_{\mu \in \mathcal{X}} 
			\int f - \langle \dualu, p \rangle \, d\mu + \langle \dualu, p(X) \rangle
			\label{e:minimax-3}
			\\
			&= 
			\adjustlimits
			\sup_{\dualu \in \R^N}
			\inf_{Z \in \dom f}
			\{f(Z) - \langle \dualu, p(Z)\rangle \} + \langle \dualu, p(X) \rangle
			\label{e:minimax-4}
			\\
			&= 
			\sup_{\substack{\dualu\in \R^N\\ \dualv \in \R\\}}
			\left\{v + \langle \dualu, p(X) \rangle :\; \dualv \leq f(Z) - \langle \dualu, p(Z)\rangle
			\quad \forall Z\in\R^{m\times n}\right\}.
			\label{e:minimax-5}
		\end{align}
	\end{subequations}
	
	The equality in \cref{e:minimax-1} is a restatement of Theorem~\ref{th:pc-measure-lp}, the equality in \cref{e:minimax-2} is true because the inner maximization is infinite unless $\int p \, d\mu = p(X)$, and the equality in \cref{e:minimax-5} is immediate from the definition of the infimum $\inf_{Z \in \dom f}\{f(Z) - \langle \dualu, p(Z)\rangle \}$ for every $u\in\R^N$.
	
	To prove the equality in \cref{e:minimax-4}, recall that $ \varphi_{\dualu}$ 
	attains a minimum at some matrix $Z_{\dualu}$. From the inequality 
	$\varphi_{\dualu}(Z_{\dualu})\leq \varphi_{\dualu}(Z)$ we conclude that 
	$\varphi_{\dualu}(Z_{\dualu})\leq \inf_{\mu \in \mathcal{X}} \int \varphi_{\dualu} \,d\mu$. The opposite 
	inequality  $\inf_{\mu \in \mathcal{X}} \int \varphi_{\dualu} \,d\mu \leq  \varphi_u(Z_\dualu)$ is true 
	by choosing $\mu=\delta_{Z_{\dualu}}$, so the claimed equality holds.
	
	There remains to establish the equality in \cref{e:minimax-3}. We do this by applying a minimax theorem of Brezis, Stampacchia and Nirenberg \cite{Brezis1972} to the function $\Phi:\mathcal{X}\times \R^N \to \R$ defined by
	\begin{equation*}
		\Phi(\mu,\dualu):= \int f - \langle \dualu, p \rangle \, d\mu + \langle \dualu, p(X) \rangle.
	\end{equation*}
	We use a version of the theorem stated in \cite[Theorem~5.2.2]{Nirenberg2001}, which if 
	$\dualu$ belongs to a finite-dimensional space requires checking the following conditions:
	\begin{enumerate}[(BNS1), noitemsep, topsep=0pt, labelwidth=\widthof{(BNS4)}, leftmargin=!]
		\item\label{bns-0} $\mathcal{X}$ is a convex subset of a topological vector space.
		\item\label{bns1} For every $\mu\in \mathcal{X}$, the function $\dualu\mapsto \Phi(\mu,\dualu)$ is upper semicontinuous on $\R^N$ and has convex superlevel sets.
		\item\label{bns2} For every $\dualu \in \R^N$, the function $\mu\mapsto \Phi(\mu,\dualu)$ is lower semicontinuous and on $\mathcal{X}$ and has convex sublevel sets.
		\item\label{bns3} There exists $\dualu_0 \in \R^N$ and $\kappa > \sup_{\R^N} \inf_{\mathcal{X}} \Phi$ such that the set $\{\mu\in\mathcal{X}:\;\Phi(\mu,\dualu_0) \leq \kappa \}$ is compact.
	\end{enumerate}
	
	Condition~\ref{bns-0} holds because $\mathcal{X}$ is a convex subset of the space of Borel signed
	measures on $\R^{m\times n}$ endowed with the weak topology. This is the locally convex topology 
	of the dual pair of the space of Borel signed measures and the space $C_b(\R^{m\times n})$ of bounded 
	continuous real functions with the bilinear form $(\mu,f)\mapsto\int f d\mu$. 
	
	Condition~\ref{bns1} holds because the function $u\mapsto \Phi(\mu,\dualu)$ is linear for every $\mu \in \mathcal{X}$.
	
	Next, we verify condition~\ref{bns2}. First, for fixed $\dualu\in\R^N$ the function 
	$\Phi(\cdot,\dualu)$ has convex sublevel sets because it is linear. To prove that it 
	is weakly lower semicontinuous, recall that $\varphi_{\dualu}$ is bounded below 
	and lower semicontinuous. Then, for every $k\in\N$, the function 
	$\varphi_{\dualu}^k: \R^{m\times n} \to\R$, $Z\mapsto\min(k,\varphi_{\dualu}(Z))$ 
	is bounded and lower semicontinuous. It follows from the Portmanteau theorem
	\cite[Theorem~8.1]{Topsoe1970}, that 
	$\Phi_k(\mu):=\int \varphi_{\dualu}^k d\mu$ is a weakly lower semicontinuous map on 
	the convex cone of finite nonnegative Borel measures on $\R^{m\times n}$. Clearly, 
	the pointwise supremum $\sup_k\Phi_k$ is weakly lower semicontinuous on the same 
	domain and it restricts to a weakly lower semicontinuous function on $\mathcal{X}$, 
	which coincides with $\Phi(\cdot,\dualu)$ by monotone convergence.
	\par
	To conclude, we verify condition \ref{bns3}. As $f$ is nonnegative, the number 
	$\kappa := 1 + f(X)$
	is strictly positive. It also satisfies the strict inequality 
	$\kappa > \sup_{\dualu}\inf_{\mu}\Phi(\mu,\dualu)$ because of 
	the inequality $f(X)\geq\fpc(X)$,
	the equalities \cref{e:minimax-1}--\cref{e:minimax-2},
	and the elementary inequality 
	$\inf_{\mu}\sup_{\dualu}\Phi(\mu,\dualu)\geq\sup_{\dualu}\inf_{\mu}\Phi(\mu,\dualu)$. 
	We now choose $\dualu_0=0$ and verify that the set
	\[
	\mathcal{K}:=\left\{\mu\in\mathcal{X}:\;\Phi(\mu,0)= \int f d\mu \leq \kappa \right\}
	\] 
	is weakly closed and tight, hence weakly compact by Prokhorov's theorem 
	\cite{Billingsley1999}. 
	
	Weak closedness follows from the identity  $\mathcal{K}=\left\{\mu\in\Pr(\dom f):\;\int fd\mu\leq \kappa \right\}$ because the set on the right hand side is weakly closed. To see this, note that $\dom f$ is closed, so the Portmanteau theorem implies that $\Pr(\dom f)$ is a weakly closed subset of $\Pr(\R^{m\times n})$, which in turn is a weakly closed set in the space of signed measures on $\R^{m\times n}$ \cite{Varadarajan1958}. Then, $\left\{\mu\in\Pr(\dom f):\;\int fd\mu\leq \kappa \right\}$ is weakly closed because for $u=0$ we have $\int fd\mu = \sup_k\Phi_k(\mu)$ and, as already shown above, this supremum is weakly lower semicontinuous on the space of finite nonnegative measures on $\R^{m\times n}$.
	
	To show tightness, we need to verify that for every $\varepsilon > 0$ there exists a compact set 
	$S_\varepsilon$ in $\R^{m\times n}$ such that 
	$\mu(\R^{m\times n}\setminus S_\varepsilon) \leq \varepsilon$ for all $\mu \in \mathcal{K}$.  
	The sublevel set $S_\varepsilon:=\{Z\in\R^{m\times n}\colon f(Z)\leq \frac{\kappa}{\varepsilon}\}$ 
	is a suitable choice because it is compact and 
	$\mu(\R^{m\times n}\setminus S_\varepsilon)
	\leq\frac{\varepsilon}{\kappa}\int fd\mu
	\leq\varepsilon
	$ for all $\mu\in\mathcal{K}$.
\end{proof}
Finally, we show that lower semicontinuity and polynomial growth of a function $f$ are inherited by its polyconvex envelope $\fpc$. These properties are highly desirable in calculus of variation problems.
\begin{corollary}
	Let $f:\R^{m \times n} \to \R \cup \{+\infty\}$ be a lower semicontinuous function 
	such that $f(X_k)/|X_k|^{\min(m,n)}\to +\infty$ whenever $|X_k|\rightarrow \infty$.
	Then $\fpc$ restricted to $\dom f$ is lower semicontinuous. If there are reals 
	$r,c_1,c_2$ such that $r>\min(m,n)$, $c_1>0$, and 
	\[
	f(X)\geq c_1(|X|^r+c_2)
	\quad
	\text{for every $X\in\R^{m\times n}$,}
	\]
	then $\fpc(X)\geq c_1(|X|^r+c_2)$ for every $X\in\R^{m\times n}$ and $\fpc$ attains 
	a minimum on $\dom f$.
\end{corollary}
\begin{proof}
	The lower semicontinuity of $\fpc$ restricted to $\dom f$ follows because $\fpc$ is the supremum of affine functions by Theorem~\ref{th:duality-improved}. The second assertion is 
	true because the function $X\mapsto c_1(|X|^r+c_2)$ is a polyconvex lower bound 
	to $f$ for every $r\geq 1$ and $c_1>0$. This and the lower semiconitinuity of 
	$\fpc$ restricted to $\dom f$ show that $\fpc$ has compact sublevel sets and 
	hence attains a minimum on $\dom f$.
\end{proof}

\subsection{Semidefinite approximations}
\label{sec:SDP-approx}

We now explain how the minimization problem in \eqref{primal} and its dual maximization 
problem in \eqref{dual} can be approximated by a hierarchy of SDPs 
under the following assumption.

\begin{assumption}\label{ass:polydata}
	The function $f:\R^{m\times n} \to \R \cup \{+\infty\}$ is polynomial on its effective domain. Moreover, $\dom f$ is a basic semialgebraic set, that is, there exist polynomials $g_1,\ldots,g_s$ such that 
	$$\dom f = \{X\in \R^{m \times n}:\; g_1(X)\geq 0,\ldots, g_s(X)\geq 0\}.$$
\end{assumption}

Thanks to this assumption, we can avoid discretizing \eqref{primal} into a finite-dimensional nonlinear program over the atoms and weights of atomic measures, as is standard in the numerical polyconvexification literature \cite{b05,mppw24}. Instead, we fix a polynomial basis to express $f$ and $\bp$ and view \eqref{primal} as a minimization problem over the integrals of the basis elements with respect to the probability measure $\mu$, that is, over the moments of $\mu$. From this perspective,  \eqref{primal} is a \emph{generalized moment problem} in the sense of Lasserre \cite{lasserre10,lasserre15}. The difficulty, of course, lies in describing the set of moments of nonnegative measures, called the \emph{moment cone}, but this can be achieved using standard techniques for polynomial optimization. The same techniques can be used to approximate the dual maximization problem in \eqref{dual}, where the challenge is to ensure the nonnegativity of the polynomial $f(Z)-v-\langle u, p(Z)\rangle$ on the semialgegbraic set $\dom f$.

\subsubsection{Moment-SOS hierarchy}\label{sec:mom-sos-envelope}

Let $d = \deg f$ be the degree of $f$.
Let $\N^{mn}_d:=\{\ba \in \N^{mn} : \sum_k a_k \leq d\}$, let $\R[Z]_d$ denote  the finite-dimensional vector space of polynomials of degree up to $d$ in the entries of $Z \in \R^{m\times n}$, and let $(b_{\ba}(Z))_{\ba \in \N^{mn}_d}$ be a basis for this space. Express the polynomials $$f(Z) = \sum_{\ba \in \N^{mn}_d} f_{\ba} b_{\ba}(Z)
\quad\text{and}\quad
\bp(Z) = \sum_{\ba \in \N^{mn}_d} \bp_{\ba} b_{\ba}(Z)$$
as linear functions of their coefficients $f_{\ba} \in \R$, ${\bp}_{\ba} \in \R^{\np}$. 
Given any sequence $\by=(y_{\ba})_{\ba \in \N^{mn}}$, we can define the linear functionals $f \mapsto \ell_{\by}(f):=\sum_{\ba} f_{\ba} y_{\ba}$ and $\bp \mapsto \ell_{\by}(\bp):=\sum_{\ba} \bp_{\ba} y_{\ba}$. Note that the functionals $\by\mapsto \ell_{\by}(f)$ and $\by\mapsto \ell_{\by}(p)$ are also linear.
Given a Borel measure $\mu$, its moment of degree $\ba$ is the real number
\begin{equation}\label{moment}
	y_{\ba} := \int b_{\ba}(Z) d\mu(Z).
\end{equation}

Let $d \in \N$ denote the degree of $f$. The cone of moments of degree up to $d$ on $\dom f$, denoted by $\mathscr{M}_d(\momsupport)$ and hereafter called simply the moment cone, is the set of vectors $(y_{\ba})_{\ba \in \N^{mn}_d}$ such that \eqref{moment} holds for some Borel measure $\mu$ supported on $\momsupport$. It is a convex cone.

Since $f$ and $\bp$ are polynomials,  problem \eqref{primal} can be reformulated as a  linear optimization problem over the moment cone, namely,
\begin{equation}\label{primalmoment}
	\fp(X) =\inf_{\by \in \mathscr{M}_d(\momsupport)}
	\ell_{\by}(f) \text{ s.t. } \ell_{\by}(1) = 1, \:\ell_{\by}(\bp) = \bp(X),
\end{equation}
where the constraint $\ell_{\by}(1) = 1$ indicates that $\by$ is the vector of moments of a probability measure.
The dual problem \eqref{dual} can of course be restated as
\begin{equation}\label{dual-v2}
	\fd(X) = \sup_{(\dualu,\dualv)\in\R^N\times\R} \dualv + \langle \dualu, \bp(X) \rangle \quad\mathrm{s.t.}\quad f - \dualv - \langle \dualu, \bp \rangle \in \mathscr{P}_d(\momsupport)
\end{equation}
where $\mathscr{P}_d(\momsupport)$ is the convex cone of nonnegative polynomials of degree up to $d$. This is the topological dual to $\mathscr{M}_d(\momsupport)$, see e.g. \cite[Thm. 8.1.2]{nie23}. While the cones $\mathscr{M}_d(\momsupport)$ and $\mathscr{P}_d(\momsupport)$ cannot be handled directly in computations, tractable semidefinite-representable outer approximations of $\mathscr{M}_d(\momsupport)$ and inner approximations of $\mathscr{P}_d(\momsupport)$ can be constructed in the following way. 

Fix $k \in \N$ such that $2k \geq \max\{\deg f, m, n\}$. Let $\Sigma_k[Z]$ be the convex cone of polynomials of $Z$ that can be written as sums of squares (SOS) of polynomials of degree $k$. The \emph{truncated quadratic module} of order $k$ associated to the basic semialgebraic set $\dom f$ is defined as
\[
\mathscr{Q}^k_d(\momsupport):=
\left\{p \in \R[Z]_d: 
p=\sigma_0 + \sum_{i=1}^s g_i\sigma_i, \; 
\sigma_0,\ldots,\sigma_s \in \Sigma_k[Z], \;
\deg(g_i \sigma_i)\leq 2k
\right\}.
\] 
It is a subset of $\mathscr{P}_d(\momsupport)$ and a convex cone. As such, it has a dual cone, which is called the \emph{pseudo-moment cone} and is given by
\begin{equation*}
	\begin{array}{lll}
		\mathscr{R}^k_d(\momsupport):=
		\big\{(y_{\ba})_{\ba \in \N^{mn}_d} : 
		&\ell_{\by}(\sigma_0) \geq 0 &\forall \sigma_0 \in \Sigma_k[Z], \\
		&\ell_{\by}(g_i \sigma_i) \geq 0 &\forall \sigma_i \in \Sigma_{k}[Z]\text{ s.t. }\deg(g_i \sigma_i)\leq 2k\big\}.
	\end{array}
\end{equation*}
Note that $\mathscr{M}_d(\momsupport) \subset \mathscr{R}^k_d(\momsupport)$ for all $k$.
Since the cone of sum-of-squares polynomials is semidefinite-representable, linear optimization problems over the pseudo-moment cone and the truncated quadratic module can be reformulated as semidefinite programs, for which powerful interior-point algorithms are available. We refer readers to \cite{lasserre10,lasserre15} for details of this reformulation. Motivated by these observations, we then relax problem \eqref{primalmoment} into the semidefinite-representable problem
\begin{equation}\label{mom}
	\fp^k(X) = \inf_{\by \in \mathscr{R}^k_d(\momsupport)} \ell_{\by}(f) \text{ s.t. } \ell_{\by}(1) = 1, \: \ell_{\by}(\bp) = \bp(X),
\end{equation}
and at the same time strengthen \cref{dual-v2} into the semidefinite-representable problem
\begin{equation}\label{sos}
	\fd^k(X) = \sup_{(\dualu,\dualv)\in\R^N\times\R} \dualv + \langle \dualu, \bp(X) \rangle \quad\mathrm{s.t.}\quad f - \dualv - \langle \dualu, \bp \rangle \in \mathscr{Q}^k_d(\momsupport).
\end{equation}
These two problems are dual to each other and, in particular, $\fd^k(X) \leq \fp^k(X)$. Both of these values are lower bound on the polyconvex envelope $\fpc(X)$ by construction. Theorem~\ref{convergence} below guarantees that they are nondecreasing in $k$ and, crucially, converge to $\fpc(X)$ if the additional assumption holds.

\begin{assumption}[Archimedean condition]\label{ass:archimedean}
	There exist $k' \in \R$ and  $\sigma_0,\ldots,\sigma_s \in \Sigma_{k'}[Z]$ such that the set $\{Z\in\R^{m \times n}:\; \sigma_0(Z) + g_1(Z) \sigma_1(Z) + \cdots + g_s(Z) \sigma_s(Z) \geq 0\}$ is compact.
\end{assumption}

\begin{remark}
	This assumption implies that $\dom f$ is compact, but is not equivalent to it (indeed, it depends explicitly on the polynomials used to define $\dom f$). If $\dom f$ is compact, then the assumption is mild because, in principle, one can add the inequality $g_{s+1}(Z) := c^2 - |Z|^2 \geq 0$ to the definition of $\dom f$ for a constant $c$ large enough not to change that set. Then, Assumption~\ref{ass:archimedean} holds with $k'=2$, $\sigma_0=\cdots=\sigma_s=0$, and $\sigma_{s+1}=1$.
\end{remark}

\begin{theorem}\label{convergence}
	For every $X \in \R^{m\times n}$ and every integer $k\geq 2d$,
	\begin{gather*}
		\fp^k(X) \leq \fd^k(X),\\
		\fp^{k}(X)\leq \fp^{k+1}(X)\leq \fpc(X), \\ 
		\fd^k(X)\leq\fd^{k+1}(X)\leq \fpc(X).
	\end{gather*}
	Moreover, if Assumption~\ref{ass:archimedean} holds, then
	$$\lim_{k\to\infty}\fp^{k}(X) = \lim_{k\to\infty}\fd^{k}(X)=\fpc(X).$$
\end{theorem}

\begin{remark}
	The values $\fp^k(X)$ and $\fd^k(X)$ always have well-defined limits $\fp^\infty(X)$ and $\fd^\infty(X)$. Assumption~\ref{ass:archimedean} suffices to ensure that $\fd^\infty(X) = \fp^\infty(X) = \fpc(X)$, but is by no means necessary: these equalities hold also for all examples in Section~\ref{sec:numerics} where $\dom f$ is not compact.
\end{remark}

\begin{proof}
	The inequality $\fd^k(X) \leq \fp^k(X)$ states the weak duality between problems~\eqref{mom} and~\eqref{sos}. The other inequalities are true because of the inclusions $\mathscr{R}^{k+1}_d(\momsupport) \subset \mathscr{R}^k_d(\momsupport)$ and $\mathscr{Q}^k_d(\momsupport) \subset \mathscr{Q}^{k+1}_d(\momsupport)$, which follow directly from the definition of these cones.
	
	To prove convergence under Assumption~\ref{ass:archimedean}, we show that for every $\varepsilon>0$ there exists $k_\varepsilon$ such that $\fd^{k_\varepsilon}(X)\geq \fpc(X) - \varepsilon$. Consider a pair $(\dualu,\dualv) \in \R^N \times \R$ that is admissible for problem \eqref{dual} and satisfies $\dualv + \langle u,p(X)\rangle \geq \fpc(X)-\frac12\varepsilon$. Then, the pair $(\dualu', \dualv'):=(\dualu, \dualv-\frac12\varepsilon)$ satisfies $\dualv' + \langle u',p(X)\rangle \geq \fpc(X)-\varepsilon$ as well as
	\begin{equation*}
		f(Z)-\dualv' - \langle u',p(Z)\rangle \geq \frac\varepsilon2 > 0
		\quad \forall Z \in \dom f.
	\end{equation*}
	This strict inequality, together with  Assumption~\ref{ass:archimedean}, means that we can invoke Putinar's Positivstellensatz \cite[Lemma~4.1]{Putinar1993} to conclude that the polynomial $f-\dualv' - \langle u',p\rangle$ belongs to $\mathscr{Q}^{k_\varepsilon}_d(\momsupport)$ for some integer $k_\varepsilon$. This implies $\fd^{k_\varepsilon}(X) \geq \dualv' + \langle u',p(X)\rangle \geq \fpc(X)-\varepsilon$.
\end{proof}

\subsubsection{Finite convergence}\label{sec:finite-convergence}

Assumption~\ref{ass:archimedean} suffices to ensure \emph{a priori} that the optimal values $\fp^k(X)$ and $\fd^k(X)$ converge to $\fpc(X)$ for every matrix $X \in \dom f$. Here, we discuss complementary sufficient conditions to determine \emph{a posteriori} if $\fp^k(X)=\fpc(X)$ for some $k$. Importantly, these conditions for finite convergence do not require $\dom f$ to be compact.

The first natural sufficient condition for finite convergence is that $\fp^k(X)=f(X)$, in which case one also concludes that $f$ is polyconvex at $X$.

\begin{proposition}\label{match}
	If the optimal value of the moment relaxation \eqref{mom} satisfies $\fp^{k^*}(X) = f(X)$ for some integer $k^*$, then $\fp^{k}(X)=\fpc(X)=f(X)$ for all $k \geq k^*$ and $f$ is polyconvex at $X$.
\end{proposition}

\begin{proof}
	Theorem~\ref{convergence} ensures that $\fp^{k^*}(X) \leq \fp^{k}(X) \leq \fpc(X) \leq f(X)$ for all $k \geq k^*$. If $\fp (X)=f(X)$, all inequalities must be equalities. 
\end{proof}

A second sufficient condition for the finite convergence of the moment-SOS hierarchy is as follows. Given a pseudo-moment vector $\by \in \mathscr{R}^k_d(\momsupport)$, let us define the \emph{moment matrix} of order $k$ as the symmetric matrix $M_k(\by)$ associated to the quadratic form $u \in \R[Z]_k \mapsto \ell_{\by}(u^2) \in \R$. Equivalently
\begin{equation}
	M_k(\by) := \ell_{\by}(b_{\ba_1}b_{\ba_2})_{\ba_1,\ba_2 \in \N^{mn}_k}
\end{equation}
where $\ell_{\by}(.)$ is meant to act element-wise on a matrix. Let us also define
\begin{equation*}
	d_g := \max_{i\in\{1,\ldots,s\}} \deg g_i
\end{equation*}
where $g_1,\ldots,g_s$ are the polynomials used to define $\det f$ (cf. Assumption~\ref{ass:polydata}).

\begin{proposition}[Flat extension]\label{flatextension}
	Let $\by^* \in \mathscr{R}^{k^*}_d(\momsupport)$ be an optimal solution of the moment relaxation \eqref{mom} of order $k^*$. Suppose $r:=\rank M_{k^*}(\by^*) = \rank M_{k^*-d_g}(\by^*)$. Then,
	$\fp   ^{k}(X)=\fpc(X)$ for all $k \geq k^*$ and
	$\by^*$ is the vector of moments of an atomic measure supported on matrices $X_1,\ldots,X_r \in \momsupport$.
\end{proposition}

\begin{proof}
	The moment matrix $M_{k^*}(\by^*)$ is positive semidefinite because $\by^* \in \mathscr{R}^{k^*}_d(\momsupport)$. Since $r:=\rank M_{k^*}(\by^*) = \rank M_{k^*-d_g}(\by^*)$ by assumption, we can then apply \cite[Thm.~5.33]{Laurent2009} to conclude that $\by^*$ is the vector of moments of an atomic measure supported on $\dom f$. Specifically, there exist matrices $X_1,\ldots,X_r \in \momsupport$ and positive weights $w_1,\ldots,w_r > 0$ with $\sum_{i=1}^r w_i = 1$ such that $\by^*  = \sum_{i=1}^s w_i \bb(X_i)$ where $\bb:=(b_{\ba})_{\ba \in \N^{mn}_{2k^*}}$. Now, using the linearity of the functional $\by\mapsto \ell_{\by}(p)$ we conclude that the atomic measure $\mu = \sum_{i=1}^r w_i \delta_{X_i}$ satisfies
	\begin{equation*}
		\int p \,d\mu = \sum_{i=1}^r w_i p(X_i) 
		= \sum_{i=1}^r w_i\ \ell_{\bb(X_i)}(p)
		=\ell_{\by^*}(\bp) = p(X),
	\end{equation*}
	so it is feasible for problem \eqref{primal}. Then, since the map $\by\mapsto \ell_{\by}(f)$ is also linear,
	\begin{equation*}
		\fpc(X) \leq 
		\sum_{i=1}^r w_i\ f(X_i) =
		\sum_{i=1}^r w_i\ \ell_{\bb(X_i)}(f) = \ell_{\by^*}(\bp) = \fp^{k^*}(X) \leq \fpc(X).
	\end{equation*}
	All inequalities must evidently be equalities. We also conclude that $\fp^k(X) \leq \fpc(X)$ for all $k \leq k^*$ because $\fp^{k^*}(X) \leq \fp^k(X) \leq \fpc(X)$.
\end{proof}

Finally, we show that our semidefinite programming approach is exact for sufficiently large relaxation order $k$ if $f$ is a lifted SOS polyconvex polynomial (cf. Definition~\ref{def:lifted-sos-pc}). Of course, since lifted SOS polyconvexity can be verified directly by solving a semidefinite program, there is no need to polyconvexify $f$ in this case. The result, however, remains of theoretical interest.

\begin{proposition}\label{prop:sospolyconvex}
	Let $f:\R^{m\times n} \to \R$ be a lifted SOS polyconvex polynomial. Let $d_g$ be the smallest positive even integer for which there exists an SOS-convex polynomial $g$ such that $f=g\circ p$. Then, $\fp^k(X) = f(X)$ for all integers $k$ such that $2k \geq \max\left\{\deg f, d_g + \min\{m,n\} \right\}$.
\end{proposition}

\begin{proof}
	Fix $X \in \R^{m\times n}$.
	Since $f$ is lifted SOS polyconvex, there exists an SOS-convex polynomial $g$ such that
	$f=g \circ \bp$. We choose $g$ to have the minimum degree, which is no larger than $d_g\geq 1$.
	
	Let $l = d_g/2$. Let $\zeta = \smash{(\zeta_a)_{a \in \N^{N}_{2l}}}$ be a basis for the space of $N$-variate polynomials of degree $2l$. For every vector $\smash{\bz = (z_a)_{a \in \N^{N}_{2l}}}$, we define a linear operator on this polynomial space by $\ell_{\bz}(\zeta) = \bz$.
	
	For every integer $k$ chosen as in the statement of the proposition, the moment relaxation \eqref{mom} satisfies
	\begin{equation*}
		\fp^k(X) 
		=
		\inf_{\substack{
				\by \in \mathscr{R}^k_d(\R^{m \times n})\\
				\bz \in \mathscr{R}^{l}_{d_g}(\R^{m \times n})
		}} \ell_{\by}(f) 
		\quad\text{s.t.}\quad
		\begin{cases}
			\ell_{\by}(1) = 1, \; \ell_{\by}(\bp) = \bp(X),\\ 
			\ell_{\bz}(1) = 1, \; \ell_{\by}(\zeta \circ \bp) = \ell_{\bz}(\zeta).
		\end{cases}
	\end{equation*}
	This is because, for every admissible $\by \in \mathscr{R}^k_d(\R^{m \times n})$, the vector $z=\ell_{\by}(\zeta\circ \bp)$ is admissible: the equality $\ell_{\by}(\zeta \circ \bp) = \ell_{\bz}(\zeta)$ holds by definition of $\ell_{\bz}$, and $\bz \in \mathscr{R}^{k'}_{d_g}(\R^{m \times n})$ because the constraint $\ell_{\by}(\zeta \circ \bp) = \ell_{\bz}(\zeta)$  implies that $\ell_z(\sigma)=\ell_y(\sigma\circ p)$ for all polynomials $\sigma \in \R[q]_{2l}$, so in particular
	$\ell_{\bz}(\sigma) = \ell_{\by}(\sigma \circ p) \geq 0$ for all degree-$2l$ SOS polynomials $\sigma$. The inequality $\ell_{\by}(\sigma \circ p) \geq 0$ holds by definition of $\mathscr{R}^k_d(\R^{m \times n})$ since the polynomial $Z\mapsto \sigma(\bp(Z))$ is a sum of squares of degree $2k$.
	
	Next, we use the identities $\ell_{\by}(f) = \ell_{\by}(g \circ \bp) = \ell_{\bz}(g)$ and $\ell_{\by}(p) = \ell_{\by}(\Id \circ p) = \ell_{\bz}(\Id)$ to rewrite 
	\begin{equation}\label{e:pwlb}
		\fp^k(X) 
		=
		\inf_{\substack{
				\by \in \mathscr{R}^k_d(\R^{m \times n})\\
				\bz \in \mathscr{R}^{l}_{d_g}(\R^{m \times n})
		}} \ell_{\bz}(g) 
		\quad\text{s.t.}\quad
		\begin{cases}
			\ell_{\bz}(1) = 1, \; \ell_{\bz}(\Id) = \bp(X)\\ 
			\ell_{\by}(1) = 1, \; \ell_{\by}(\zeta \circ \bp) = \ell_{\bz}(\zeta).
		\end{cases}
	\end{equation}
	Now, since $g$ is SOS-convex we can use an extension of Jensen's inequality proved in \cite[Thm.~2.6]{lasserre09} (see also \cite[Sec.~7.1.2]{nie23}) to conclude that $\ell_{\bz}(g) \geq g(\ell_{\bz}(\Id))$ for all pseudo-moment vectors $\bz$. Minimizing both sides of this inequality over $z$ that are admissible for the right-hand side of \eqref{e:pwlb} reveals that $\fp^k(X) \geq g(p(X))$. But $\fp^k(X) \leq f(X) = g(p(X))$ by Theorem~\ref{convergence}, so $\fp^k(X)= g(p(X))=f(X)$.
\end{proof}
 
\subsection{Numerical examples}\label{sec:numerics}

We now illustrate our semidefinite programming approach to computing polyconvex envelopes on a range of examples.

\subsubsection{A basic polyconvex but not convex example}

For example if $m=n=2$, letting
\begin{equation}\label{eq:x}
X=\left(\begin{array}{cc}x_1&x_3\\x_2&x_4\end{array}\right)
\end{equation}
we define the moment of $\mu$ of order $\ba \in \N^4$ as follows:
\[
y^{\ba}:=\int X^{\ba} d\mu(X)=\int {x_1}^{a_1} {x_2}^{a_2} {x_3}^{a_3} {x_4}^{a_4} d\mu(X).
\]
Let $f(X):=(\det\,X)^2$, i.e.
\[
f(X) = \sum_{\ba \in \N^4_4} f_{\ba} \bx^{\ba} = x^2_1x^2_4-2x_1x_2x_3x_4+x^2_2x^2_3
\]
and
\[
\ell_{\by}(f) = y_{2002}-2y_{1111}+y_{0220}.
\]
Linear problem \eqref{primalmoment} then reads
\[
\begin{array}{rcl}
\fp(X) & := & \inf_{\by \in \mathscr{M}_4(\R^4)} y_{2002}-2y_{1111}+y_{0220} \\& \mathrm{s.t.} & (y_{0000},\:y_{1000},\:y_{0100},\:y_{0010},\:y_{0001},\:y_{1001}-y_{0110}) = (1,x_1,x_2,x_3,x_4,x_1x_4-x_2x_3)
\end{array}
\]
where the minimization is in $\mathscr{M}_4(\R^4)$, the cone of moments of 4 variables of degree up to 4, which has dimension $70$.

The moment relaxation \eqref{mom} is constructed with the following Matlab script which uses the GloptiPoly interface: 
\begin{verbatim}
mpol Z 2 2
k = 2; % relaxation order 
f = det(Z)^2; % function
X = randn(2); % polyconvex envelope at X
P = msdp(min(f), mom([1;Z(:);det(Z)])==[1;X(:);det(X)], k);
[stat,obj] = msol(P);
disp(['lower bound (relaxation) = ' num2str(obj)])
disp(['upper bound (Dirac at X) = ' num2str(det(X)^2)]) 
\end{verbatim}
It is then solved by any semidefinite optimization solver, e.g. MOSEK \cite{MOSEK}.

We observe numerically that the first moment relaxation  (order $k=2$) always converges: it returns a lower bound $\fp^2(X)$ which is equal to the upper bound $f(X)$ and hence equal to $\fp(X)$ for all $X \in \R ^{2\times 2}$, illustrating Proposition \ref{match}.
Note also that $f$ is lifted SOS polyconvex according to Definition \ref{def:lifted-sos-pc} with the convex function $q \in \R^5 \mapsto g(q)=q^2_5$, so that Proposition  \ref{prop:sospolyconvex} ensures that the first relaxation is exact.

The restriction of the polyconvex function $X \in \R^{2\times 2} \mapsto f(X)=(\det X)^2$ to the affine subspace
\[
X(x) := \left(\begin{array}{cc}
x & 1 \\ 1 & x \end{array}\right)
\]
parametrized by $x \in \R$ is the nonconvex double well function $x \in \R \mapsto f(X(x)) = (x^2-1)^2$.
The function
\[
\begin{array}{rcl}
x & \mapsto & \inf_{\by \in \mathscr{R}^2_4(\mathscr{K})} y_{2002}-2y_{1111}+y_{0220} \\& \mathrm{s.t.} & (y_{0000},\:y_{1000},\:y_{0100},\:y_{0010},\:y_{0001}) = (1,x,1,1,x)
\end{array}
\]
is the value of the first moment relaxation \eqref{mom} (order $k=2$) where we remove the non-linear entry $\bp$ in the moment problem. This corresponds to the convex envelope of $f$. Numerically, this function is equal to zero everywhere. The function 
\[
\begin{array}{rcl}
x & \mapsto & \inf_{\by \in \mathscr{R}^2_4(\mathscr{K})} y_{2002}-2y_{1111}+y_{0220} \\& \mathrm{s.t.} & (y_{0000},\:y_{1000},\:y_{0100},\:y_{0010},\:y_{0001},\:y_{1001}-y_{0110}) = (1,x,1,1,x,x^2-1)
\end{array}
\]
is the value of the first moment relaxation \eqref{mom} (order $k=2$) for $\bp$ equal to the minor map. Numerically, this function coincides with the double well function, illustrating the key role played by the additional moment constraint $y_{1001}-y_{0110} = x^2-1$.

\subsubsection{A polyconvex double well}

Let $f(X):=|X-I|^2|X+I|^2$ for $X \in \R^{2\times 2}$. 
Solving \eqref{primalmoment} at relaxation order 2 
for $X=0$ returns a rank-one certificate of polyconvexity corresponding to the Dirac measure at $0$, and a value of $f(0)=\fpc(0)=4$. The dual problem \eqref{dual} has solution $r_0=4$, $\br=(0,0,0,0,-8)$ corresponding to the largest polyconvex function $4-8 \det X$ touching $f$ from below at $X=0$.
Letting
$$X=\left(\begin{array}{cc}x_1&x_3\\x_2&x_4\end{array}\right)$$
a few algebraic manipulations provide the SOS optimality certificate
$$f(X) - 4 + 8\,\det X
= (x^2_1 + x^2_2 + x^2_3 + x^2_4)^2
	+ 4(x_2-x_3)^2$$
proving that indeed $f(X) \geq 4-8\det X$ for all $X \in \R^{2\times 2}$. In addition, from the expression
$$
f(X)
= 4 + 4(x_2-x_3)^2 + (x^2_1 + x^2_2 + x^2_3 + x^2_4)^2 - 8\,\det X
$$
we deduce that $f$ is a convex function of $X$ and $\det X$, so it is polyconvex. Note that this representation of $f$ is precisely \cref{e:dw-sos-convex-repr} for $n=2$. We refer to \cite{hk25} for further discussion of the polyconvexity of double-well functions.

\subsubsection{A non-polyconvex double well}

Fix $\epsilon > 0$, let
\[
X_1 = \left(\begin{array}{cc}
1 & \epsilon \\ 0 & 1
\end{array}\right) \quad\text{and}\quad
X_2 = \left(\begin{array}{cc}
1 & -\epsilon \\ 0 & 1
\end{array}\right)
\]
and consider the double-well function $f(X):=|X-X_1|^2|X-X_2|^2$.

Function $f$ is not polyconvex because $f(\id)=\epsilon^4$, whereas choosing $\mu = \frac{1}{2}(\delta_{X_1}+\delta_{X_2})$ in the minimization problem \cref{primal} for the polyconvex envelope shows that $\fp(I)\leq 0$.

This choice of $\mu$ is in fact optimal, meaning that $\fp(I)= 0$. This is certified numerically by solving the relaxation \cref{primalmoment} with relaxation order $k=2$ and verifying that the optimal solution satisfies the flat extension conditions in Proposition \ref{flatextension}.

\subsubsection{A non-polyconvex quadruple well}\label{sec:quadwell}

Let
\[
X_1 = \left(\begin{array}{cc}
1 & 0 \\ 0 & 1
\end{array}\right), \quad
X_2 = \left(\begin{array}{cc}
-1 & 0 \\ 0 & 1
\end{array}\right), \quad
X_3 = \left(\begin{array}{cc}
1 & 0 \\ 0 & -1
\end{array}\right), \quad
X_4 = \left(\begin{array}{cc}
-1 & 0 \\ 0 & -1
\end{array}\right)
\]
and $f(X):=|X-X_1|^2|X-X_2|^2|X-X_3|^2|X-X_4|^2$. For the choice
\[
X(x) := \left(\begin{array}{cc}
0 & 0 \\x & 2x \end{array}\right)
\]
we solve the moment relaxation of order $k=4$ for different values of the parameter $x \in [0,1]$. Global optimality of the relaxation is always certified by flat extension as in Prop. \ref{flatextension}.

\begin{figure}[t]
\centering
\includegraphics[width=\linewidth]{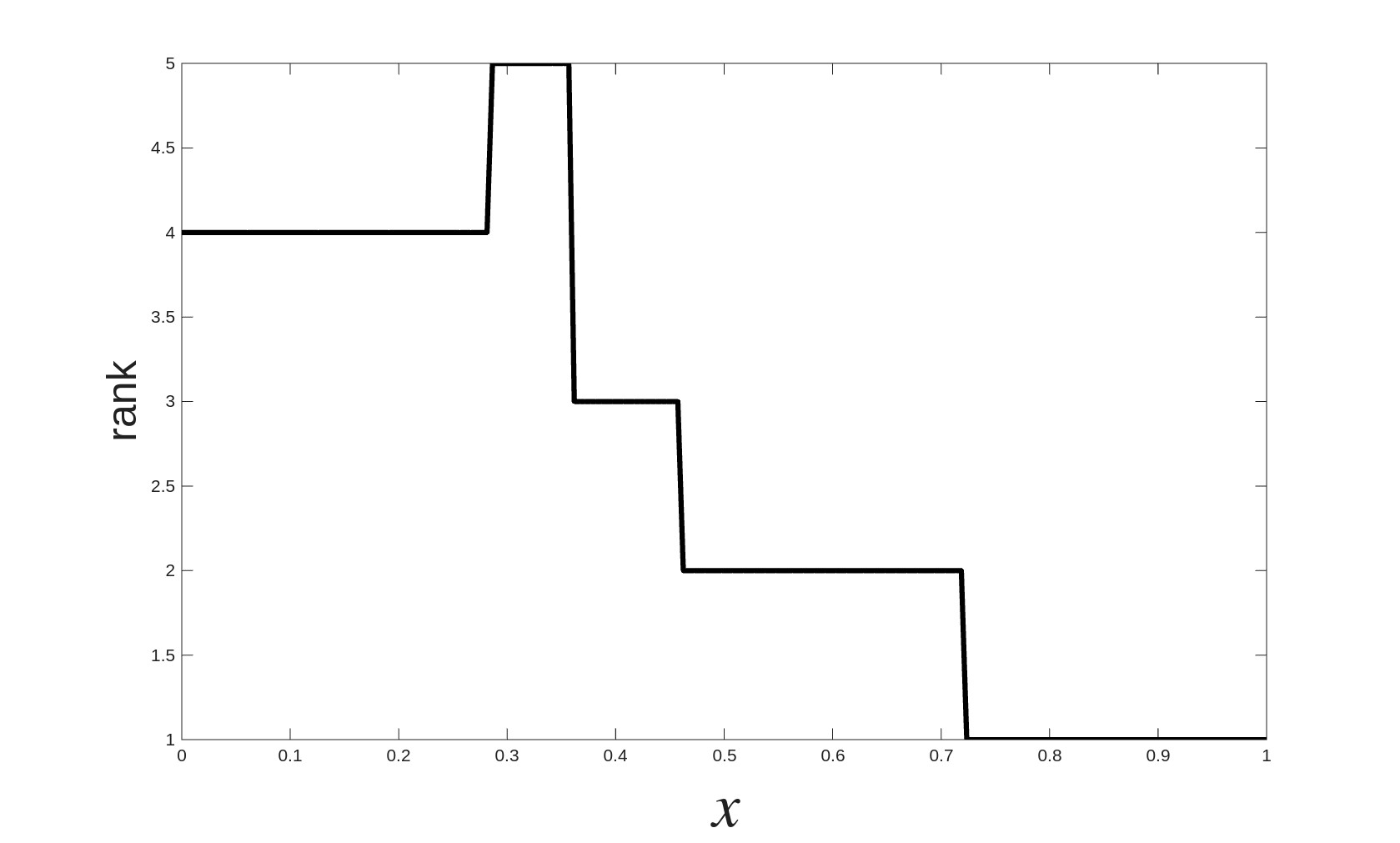}
\caption{Rank of the moment matrix for the example in Section~\ref{sec:quadwell} as a function of the parameter $x$.}
\label{fourmatrices}
\end{figure}

Figure~\ref{fourmatrices} illustrates the rank of the moment matrix as a function of $x$. When the rank is greater than one, the function $f$ is not polyconvex at $X$ and the moment relaxation returns the value of the polyconvex envelope $\fpc(X(x))$. For example if $x=0$ the 4-atomic measure solving linear problem \eqref{th:pc-measure-lp} is
\[
\mu = \frac{1}{4}(\delta_{X_1}+\delta_{X_2}+\delta_{X_3}+\delta_{X_4}).
\]
It holds
\[
f(0) = |X_1|^2|X_2|^2|X_3|^2|X_4|^2 = 16
\]
while
\[
\fpc(0) = \int f(Z)d\mu(Z) = \frac{1}{4}(f(X_1)+f(X_2)+f(X_3)+f(X_4)) = 0
\]
and
\[
\int p(Z) d\mu(Z) = p(0) = \left(0,0,0,0,0\right)
\]
with $p(X)=\left(x_1,x_2,x_3,x_4,x_1x_4-x_2x_3\right)$ and $X$ as in \eqref{eq:x}.

\subsubsection{Computational performance on a double-well function}\label{sec:twowell3d}

Let us finally illustrate the efficiency of our approach on the non-convex double-well function $X \in \R^{n\times n} \mapsto f(X):=(|X|^2-1)^2$. This is an established benchmark for analytical and computational semiconvexification, see \cite[Sec. 4.2]{mppw24}.
Its quasiconvex, polyconvex and convex envelopes coincide with the function $\max(|X|^2-1,0)^2$, so that the computational results can be compared with the analytic solution.
In \cite{mppw24}, isotropy of the function is exploited to reduce the computational burden. In our experiments, we do not perform any reduction, so the reported computation times are representative of the computational performance for any other function with the same number of variables, degree, and relaxation order $k$.

\begin{figure}[t]
\centering
\includegraphics[width=\linewidth]{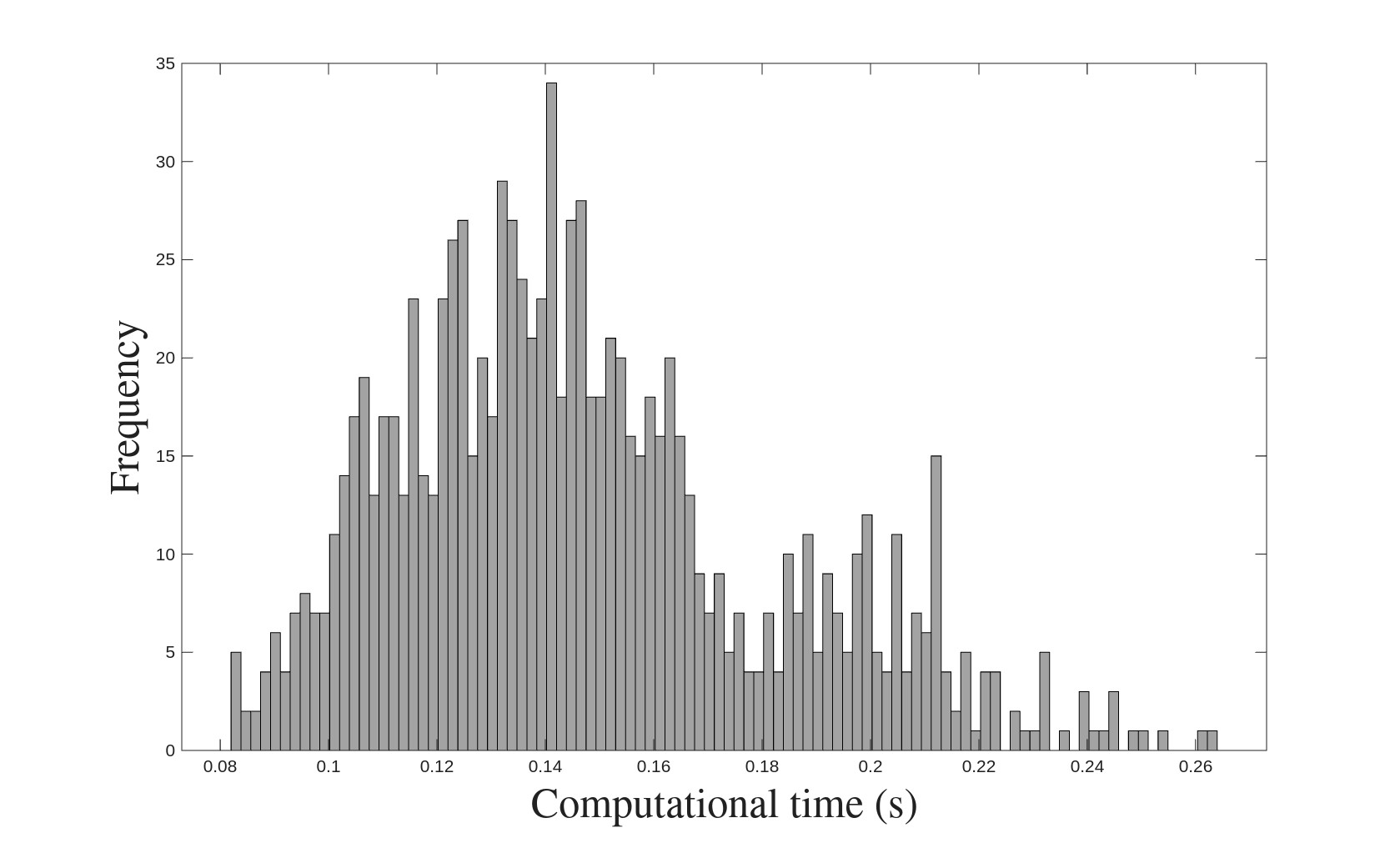}
\caption{Histogram of computational times (in seconds, on a standard laptop) for computing the polyconvex envelope of the quartic function in Section~\ref{sec:twowell3d} at 1000 randomly generated matrices. \label{fig:timingsdoublewell3d}}
\end{figure}

For $n=3$ the linear moment problem \eqref{primalmoment} has $20$ linear constraints and $715$ variables. The first moment relaxation \eqref{mom} (order $k=2$) has a positive semidefinite moment matrix constraint of size $55$.
Figure~\ref{fig:timingsdoublewell3d} shows a histogram of computational times for 1000 randomly generated points $X$, using the semidefinite solver MOSEK \cite{MOSEK} on a laptop equipped with Intel Core i7-1165G7 @ 2.80GHz processor and 16GB of RAM. For all the points, the value of the relaxation is equal to machine precision to the value of the polyconvex envelope. The mean computational time is $0.15$s and the standard deviation is $0.035$s.

\section{Conclusion}

In this work, we have exploited techniques for polynomial optimization to introduce a new perspective on polyconvexity for functions that are polynomial on their domain. Specifically, we have used SOS polynomials and moment methods to (i)~introduce new computationally tractable sufficient conditions for polyconvexity and
(ii)~efficiently evaluate polyconvex envelopes for non-polyconvex polynomials.
Our methods are easily implemented using mature software for polynomial optimization and semidefinite programming, allowing for a systematic investigation of polyconvexity with fast computations. In fact, the good performance observed in our computational examples may be improved further by exploiting algebraic structures such as sparsity \cite{lasserre2006,wang-tssos,wang-chordaltssos,wang-cstssos,yzgf-review} or symmetries \cite{Gatermann2004,Riener2013}. In particular, it should be possible to exploit invariance under the $O(n)$ and $SO(n)$ symmetry groups typical of isotropic functions.

Beyond improving computational efficiency, there are various open problems that deserve further attention. One is to better understand the expected gap between polyconvexity and SOS polyconvexity. Since there exist nonnegative polynomials that are not SOS and convex polynomials that are not SOS convex, we expect there should be polyconvex polynomials that are not SOS polyconvex in the sense of Definition~\ref{def:first-order-sos-pc}. However, we currently have no examples.

Another interesting problem is to determine convergence rates for our moment-SOS hierarchy for computing polyconvex envelopes. Since the framework we presented is a particular example of a broader class of moment problems, the general convergence rates derived in \cite{STL2026,GM2025} apply, but are likely suboptimal because they do not take into account the particular structure of polyconvexification problems. We wonder if this structure can be exploited to obtain better theoretical convergence rates, that more closely match the excellent performance observed in practice.

Finally, another interesting direction for future work would be to embed our computational approach to polyconvexification into finite-element algorithms for multi-scale nonlinear elasticity problems exhibiting microstructure. A solution strategy might be seen as a von Stackelberg game \cite{Stackelberg} where the elastic deformation is a leader and material microstructures encoded in measures representing the polyconvex envelope are followers. Our computational methods could be used to solve the subproblem for these followers, with potential to accelerate multiscale simulations in structural mechanics.

\subsection*{Acknowledgments}

This work was partly funded by the European Union/M\v SMT \v CR under the ROBOPROX project
(reg. no. CZ.02.01.01/00/22 008/0004590). It was also partly funded by the ANR-DFG project MONET (ANR-25-CE48-6598-01).

The authors acknowledge the use of AI for assistance with brainstorming, mathematical development, coding and
drafting. The final content, analysis and conclusions remain the sole responsibility of the authors.


\end{document}